\documentclass{amsart}
\usepackage{amsthm}
\usepackage{hyperref}
\newtheorem{theorem}{Theorem}
\newtheorem{definition}[theorem]{Definition}
\newtheorem{lemma}[theorem]{Lemma}
\newtheorem{proposition}[theorem]{Proposition}
\newtheorem{assumption}[theorem]{Assumption}
\newtheorem{corollary}[theorem]{Corollary}
\newtheorem{remark}[theorem]{Remark}
\newtheorem{example}[theorem]{Example}

\newcommand{\E}{\mathbb{E}}

\providecommand{\norm}[1]{\ensuremath{\lVert#1\rVert}}
\providecommand{\abs}[1]{\ensuremath{\lvert#1\rvert}}

\usepackage{color}
\definecolor{darkgreen}{rgb}{0, .5, 0}
\definecolor{darkred}{rgb}{.5, 0, 0}
\definecolor{orange}{rgb}{1, 0.3, 0.1}

\title[Multilinearity preserving property]{Independent increment processes: A multilinearity preserving property}
\date{\today}
\thanks{FEB acknowledges financial support from FINEWSTOCH, a research project funded by the Norwegian Research Council. We are grateful to the referee and associate editor for suggestions that greatly improved the presentation of the paper.}
\author{Fred Espen Benth, Nils Detering and Paul Kr\"uhner}
\address{Fred Espen Benth \\
University of Oslo\\
Department of Mathematics \\
P.O. Box 1053, Blindern\\
N--0316 Oslo, Norway}
\email[]{fredb\@@math.uio.no}
\address{Nils Detering \\ 
Department of Statistics and Applied Probability\\
CA 93106 Santa Barbara, USA}
\email[]{detering\@@pstat.ucsb.edu}
\address{Paul Kr\"uhner \\
Institute for Financial and Actuarial Mathematics \\
The University of Liverpool \\
The Mathematical Sciences Building \\
Liverpool L69 7ZL, UK}
\email[]{paul.eisenberg\@@liverpool.ac.uk}

\keywords{Infinite dimensional stochastic processes, polynomial processes, Banach algebras, conditional expectation, multilinear maps}

\begin{document}
\maketitle

\begin{abstract}
We observe a multilinearity preserving property of conditional expectation for infinite dimensional independent increment processes defined on some abstract Banach space $B$. It is similar in nature to the polynomial preserving property analysed greatly for finite dimensional stochastic processes and thus offers an infinite dimensional generalisation. However, while polynomials are defined using the multiplication operator and as such require a Banach algebra structure, the multilinearity preserving property we prove here holds even for processes defined on a Banach space which is not necessarily a Banach algebra. In the special case of $B$ being a commutative Banach algebra, we show that independent increment processes are polynomial processes in a sense that coincides with a canonical extension of polynomial processes from the finite dimensional case. The assumption of commutativity is shown to be crucial and in a non-commutative Banach algebra the multilinearity concept arises naturally. Some of our results hold beyond independent increment processes and thus shed light on infinite dimensional polynomial processes in general. 
\end{abstract}

\section{Introduction}
An $\mathbb R$-valued process $(X(t))_{t\geq 0}$ is said to be a {\it polynomial process} if for every polynomial $p$ of degree $n$, there exists another polynomial $q$  of degree at most $n$ such that $\E [ p(X(t)) \,| \mathcal{F}_s  ]=q(X(s)) $ for every $t\geq s\geq 0$. The polynomial $q$ may have deterministic time-dependent coefficients. Examples of polynomial processes in $\mathbb{R}$ are affine processes or the multidimensional Jacobi process, among others (see Ackerer, Filipovi\'c and Pulido~\cite{AFP} for an application of the 
Jacobi process to stochastic volatility). Polynomial processes with values in the Euclidean space $\mathbb R^d, d<\infty$, or subsets thereof have received much attention recently especially due to their applications in financial mathematics. We refer the reader to Cuchiero, Keller-Ressel and Teichmann~\cite{CKT}, Filipovi\'c and Larsson~\cite{LF} and Foreman and S\o rensen \cite{FS}, and the references therein for an analysis and application of these processes and an overview of the existing literature. 

In the present paper we lift the notion of polynomial processes to general Banach spaces with a particular focus on independent increment processes. We thus show in a first instance how polynomial processes can be extended to an infinite dimensional setting and shed light on the special role commutative Banach algebras play in this context. We introduce a multilinearity preserving property for processes in a general Banach space and a polynomial preserving property for processes with values in a Banach algebra. 

To explain our approach in slightly more detail, let $\mathcal{L}_n : B^n \rightarrow B$, $n\in\mathbb N$, be a multilinear map on $B^n$, the product space of $n$ copies of the Banach space $B$. We say that a $B$-valued stochastic process  $(X(t))_{t\geq 0}$ is a {\it multilinear process}
if for every $n\in\mathbb N$ and every multilinear map $\mathcal L_n$, it holds that  
$$
\E [ \mathcal{L}_n(X(t),\dots,X(t)) \,| \mathcal{F}_s  ]=\sum_{k=0}^n  \mathcal{M}_k(X(s;t),\dots,X(s;t))
$$
for all $t\geq s\geq 0$. Here, for $1 \leq k\leq n, \mathcal M_k$ are again multilinear maps, now on $B^k$ and $\mathcal{M}_0\in B$ is a constant. Further $X(s;t)$ is an $\mathcal{F}_s$-measurable random variable with values in $B$. Often this is simply $X(s)$, and moreover, the multilinear maps 
$\mathcal M_k$ may be depending (deterministically) on $t$ and $s$. 
This of course includes the representation for monomials (and by linearity also polynomials) by defining 
$\mathcal{L}_n:B^n \rightarrow B$ by $\mathcal{L}_n (x_1,\dots,x_n) = x_1\cdots x_n $ if a designated multiplication in $B$ is defined. The multilinearity property thus extends naturally the idea of the polynomial property in that {\em moment like} quantities of $(X(t))_{t\geq 0}$ 
can be easily calculated. The structure, however, does not focus on the particular moments arising from the designated multiplication operator.

Our first main Theorem~\ref{prop:indep-incr-multi} shows that independent increment processes (and variants thereof) in a Banach space are multilinearity preserving processes. If the Banach space has a multiplication defined and forms a Banach algebra, we show in Theorem~\ref{poly:preserv} and Proposition~\ref{prop:genpoly:poly} that they are also (generalized) polynomial processes. Moreover, in Proposition~\ref{prop:form} we show that in most cases, the multilinear preserving property allows to calculate conditional expectations even of multilinear forms in contrast to maps, a property crucial for applications. As auxiliary results we derive several specific properties of conditional expectations in (possibly non-commutative) Banach algebras, which might be of independent interest. 

{\bf Applications:} There is a range of possible applications for our results. We work out two of them in more detail. First, we show how the multilinearity property comes in handy to calculate conditional moments of the norm of the process. The efficient calculation of moments is important for instance in order to determine population parameters in statistical estimation. Second, we apply our results to the pricing of options on commodity forwards when the entire forward curve is modeled as an element in the Filipovi\'c space (see for instance Benth and Kr\"uhner \cite{BK}). In this setting a commutative Banach algebra can be defined by pointwise multiplication of the forward prices and we exploit the generated polynomial structure for pricing options on forwards. This leads to an easy to calculate formula for pricing European options on the forward price.

Additionally we show how processes whose values are measures can be treated, linking our analysis and definitions to the work of Cuchiero, Larsson and Svaluto-Ferro \cite{CLSF}. Examples of relevance for non-commutative Banach algebras include matrix-valued stochastic processes or more general processes of linear bounded operators where multiplication is the concatenation of operators. These cases cover infinite-dimensional stochastic volatility models (see Benth, R\"udiger and S\"uss \cite{BRS}).


The outline of the paper is as follows. In Section~\ref{sec:multilin:process} we introduce our notion of a multilinear process and prove our first main results for processes with values in a Banach space. In Section~\ref{sec:indep:incr} we restrict the state space to be a commutative Banach algebra and analyse polynomial versus multilinear processes. We pay special attention to the Ornstein-Uhlenbeck dynamics. Finally, in Section~\ref{sec:applic} we provide several possible applications of our results. The Appendix~\ref{sec:cond:exp} contains some important auxiliary results about conditional expectation in Banach spaces and algebras, which we could not find in the literature.

\section{Multilinear maps and multilinear processes}\label{sec:multilin:process}
In this section we study stochastic processes with values in a Banach space $B$ which possess certain stability properties with respect to "polynomials" and
conditional expectation. We introduce polynomials via certain multilinear maps,
that are defined next:
 
Denote by $B^k=B \times \cdots \times B$ the product space of $k\in\mathbb N$ copies of $B$ equipped with the norm  
$\|\cdot\|_k:=\sup_{1\leq i\leq k}\|\cdot\|$. The product space $B^k$ becomes again a Banach space.  
We introduce the following definition of $k$-linear maps, that will play an important role in the sequel:
\begin{definition}
\label{def:k-maps}
We say that $\mathcal{L}_k : B^k \rightarrow B$ for $k\in\mathbb N$ is a {\tt $k$-linear map} if it is linear in each argument in the sense that for any
$x_1,x_2,\ldots,x_k,y\in B$ and $a,b\in\mathbb F$ 
\begin{align*}
\mathcal{L}_k (x_1,\dots,x_{j-1},&ax_j +by ,x_{j+1},\dots , x_k) \\ &=a \mathcal{L}_k (x_1,\ldots, x_k) + b \mathcal{L}_k (x_1,\dots,x_{j-1},y ,x_{j+1},\dots, x_k)
\end{align*}
for each $j=1,\ldots,k$. A $k$-linear map $\mathcal L_k$ is {\tt bounded} if there exists a constant $K>0$ such that 
$$
\|\mathcal L_k(x_1,\ldots,x_k)\|\leq K\|x_1\|\cdots\|x_k\|
$$
for all $x_1,\ldots,x_k\in B$. We denote the space of bounded $k$-linear maps by $L_k (B)$. 
\end{definition}
Notice that $L_1(B)=L(B)$, the space of bounded linear operators. Often we will call a $k$-linear map simply multilinear without specifying the dimension.
 
A $k$-linear map $\mathcal{L}_k$ induces a {\it $k$-monomial} $\mathcal{M}_k: B \rightarrow B$ by
\begin{equation}
\mathcal{M}_k(x):=\mathcal{L}_k (x,\dots,x) .
\end{equation}
If $\mathcal L_k\in L_k(B)$, we see that $\|\mathcal M_k(x)\|\leq K\|x\|^k$, and we denote the set of all such $k$-monomials
by $M_k(B)$. Of course, $M_1(B)=L(B)$, the space of bounded operators. Additionally, we define $M_0(B):=B$ for completeness. 
$M_0(B)$ will play the role as the space of "constants", or, zero-order monomials. We remark that $M_k(B)$ is a vector space over the same field as
$B$. We have the following result showing that the monomials are locally Lipschitz continuous on $B$:
\begin{proposition}
\label{prop:local-lip-mono}
If $\mathcal M_k\in M_k(B)$, then for any $x,y\in B$
$$
\|\mathcal M_k(x)-\mathcal M_k(y)\|\leq C(\|x\|,\|y\|)\|x-y\|
$$
where $C(\|x\|,\|y\|)=K\sum_{i=1}^{k}\|x\|^{k-i}\|y\|^{i-1}$ for some positive constant $K$.
\end{proposition}
\begin{proof}
We notice that for $k=1$, $\mathcal M_1\in L(B)$ and therefore Lipschitz continuous. Let therefore $k\geq 2$. As $\mathcal M_k\in M_k(B)$,
we have for $x\in B$ that $\mathcal M_k(x)=\mathcal L_k(x,\ldots,x)$ for a bounded $k$-linear map, $\mathcal L_k\in L_k(B)$. By adding and subtracting
$\mathcal L_k(y,\ldots,y,x,\ldots,x)$, where $y\in B$ goes successively through all the $k-1$ first coordinates, we find from the triangle inequality and 
the multilinearity property of $\mathcal L_k$,
\begin{align*}
\|\mathcal M_k(x)-\mathcal M_k(y)\|&=\|\mathcal L_k(x,\ldots,x)-\mathcal L_k(y,\ldots,y)\| \\
&\leq \|\mathcal L_k(x,\ldots,x)-\mathcal L_k(y,x,\ldots,x)\| \\
&\qquad+\|\mathcal L_k(y,x,\ldots,x)-\mathcal L_k(y,y,x,\ldots,x)\| \\
&\qquad+\cdots \\
&\qquad+
\|\mathcal L_k(y,\ldots,y,x)-\mathcal L_k(y,\ldots,y)\| \\
&=\|\mathcal L_k(x-y,\ldots,x)\| \\
&\qquad+\|\mathcal L_k(y,x-y,\ldots,x)\| \\
&\qquad+\cdots \\
&\qquad+\|\mathcal L_k(y,\ldots,y,x-y)\| \\
&\leq K\|x-y\|\|x\|^{k-1}+K\|y\|\|x-y\|\|x\|^{k-2}+\ldots \\
&\qquad+K\|y\|^{k-1}\|x-y\|.
\end{align*}
The last inequality follows from the boundedness of $\mathcal L_k$. The result follows.
\end{proof}

Let $(X(t))_{t\geq 0}$ be a $B$-valued stochastic process, that is, a family of $B$-valued random variables $X(t)$ indexed by $t\geq 0$.
In the following we shall be interested in the conditional expectation $\mathcal{M}_k(X(t))$ given $\mathcal{F}_s$ for $t\geq s\geq 0$ where
$\mathcal M_k\in M_k(B)$. More specifically, we want to define and study processes $(X(t))_{t\geq 0}$ where for any 
$\mathcal M_k\in M_k(B)$ there exists a family of $j$th-order monomials $\overline{\mathcal{M}}_j \in M_j (B)$ with $j\leq k$ such that
\begin{equation}
\label{eq:multlinearprocess}
\E [\mathcal{M}_k(X(t)) \,|\,\mathcal F_s] =  \sum_{j=0}^k \overline{\mathcal{M}}_j (X(s;t)),
\end{equation}
and where $X(s;t)$ is some strongly $\mathcal F_s$-measurable random variable. As we see, we are interested in processes which preserve the "polynomial" order, as the monomials on the right hand side are not exceeding $k$ in their orders. Moreover, the $j$th-order monomials $\overline{\mathcal{M}}_j$ are allowed to depend (deterministically) on $s$ and $t$, however, we do not
state this explicitly to lessen the notational burden. 

A minimal requirement for studying \eqref{eq:multlinearprocess} is that $\mathcal{M}_k(X(t))$ is Bochner integrable. As $X(t)$ is strongly
measurable and $\mathcal M_k$ is continuous by Proposition~\ref{prop:local-lip-mono}, it follows from Lemma~\ref{lemma:cont-preserve-measur} in Appendix \ref{sec:cond:exp} that $\mathcal M_k(X(t))$ is strongly measurable.
We introduce the following assumption:
\begin{assumption}
\label{moment-assumption-X}
The process $(X(t))_{t\geq 0}$ has finite moments of all order, i.e., for any $n\in\mathbb N$, $\E[\|X(t)\|^n]<\infty$ for all $t\geq 0$. 
\end{assumption}
Since we have $\|\mathcal M_k(x)\|\leq K\|x\|^k$, it follows under Assumption~\ref{moment-assumption-X} that $\mathcal M_k(X(t))$ is Bochner integrable, and in particular the conditional expectation 
in \eqref{eq:multlinearprocess} exists. In the study of polynomial diffusions in $\mathbb{R}^k$, a moment assumption is usually not needed and follows from $\E[\|X(0)\|^n]<\infty$ for all $n$ (see Thm. 3.1. in Filipovi\'c and Larsson~\cite{LF}). For general state space and driving noise however, we will need it.

As a simple example, let us look at the case $B=\mathbb{R}$ and the function  $\mathcal{M}_k :  \mathbb{R} \rightarrow  \mathbb{R}$ given by $\mathcal{M}_k(x) = x^k$. Then one easily observes that $\mathcal{M}_k$ is induced by the $k$-linear map $\mathcal{L}_k : \mathbb{R} \times \cdots \times \mathbb{R}\rightarrow \mathbb{R}$ , $(x_1,\dots,x_k)\rightarrow x_1\cdot \dots \cdot x_k$. 
In Cuchiero et al.~\cite{CKT} and Filipovi\'c and Larsson \cite{LF}, a real-valued $\mathcal F_t$-adapted stochastic process $(X(t))_{t\geq 0}$ is called a {\it polynomial process} if for any $n\in\mathbb N$,
$\E [ (X(t))^n  \,|\,\mathcal F_s ] = q_n (X(s))$ for some polynomial $q_n$ of degree at most $n$. We will later see that $k$-linear maps arise naturally when dealing with polynomials in possibly non-commutative Banach algebras.

Next, let us define multilinear processes:
\begin{definition}\label{Multi:process}
Let $(X(t))_{t\geq 0}$ be a $B$-valued stochastic process and $(X(s;t))_{0 \leq s \leq t < \infty}$ a family of $B$-valued random variables, such that $X(s;t)$ is strongly $\mathcal{F}_s$-measurable. The process $(X(t))_{t\geq 0}$ is said to be a
{\tt multilinear process} with respect to the family $(X(s;t))_{0 \leq s \leq t < \infty}$ if for any $k\in\mathbb N$ and 
$\mathcal M_k\in M_k(B)$, 
there exists a family of $j$th-order monomials $\overline{\mathcal M}_j\in M_j(B)$, $j\leq k$,
such that for all $s\leq t$ it holds,
\begin{equation}
\label{def:multi-process-cond-exp}
\mathbb E[\mathcal{M}_k (X(t)) \,|\,\mathcal{F}_s]= \sum_{j=0}^k \overline{\mathcal{M}}_j (X(s;t)) .
\end{equation} 
\end{definition}  
Note that if we take a linear combination of monomials up to order $k\in\mathbb N$, that is, 
$\mathcal P_k:=\sum_{j=0}^k p_j\mathcal M_j$ for $p_j\in \mathbb F$ and $\mathcal M_j\in M_j(B)$ for $j=0,\ldots,k$, we find by 
the vector space structure of $M_j(B)$ that $\mathcal P_k$ can be represented by a sum of monomials up to degree $k$. Hence, we can use
a linear combination of monomials in the conditional expectation defining a multilinear process in Definition~\ref{Multi:process}. 
\begin{remark}
Instead of considering $\E [\mathcal{M}_k(X(t)) \,|\,\mathcal F_s] $ with $\mathcal M_k\in M_k(B)$ one could consider a $k$-linear form $\mathcal{L}_k : B^k \rightarrow \mathbb{F}$ and their expectation $\E [\mathcal{M}_k(X(t)) \,|\,\mathcal F_s] $. We show in Proposition~\ref{prop:form} that when $B$ is a Hilbert space and the process $(X(t))_{t\geq 0}$ is multilinear, then the representation of $k$-linear forms follows for multilinear processes. Therefore the multilinearity property is stronger and we consider it here. Moreover, if a dedicated multiplication operator exists then it is a bilinear form and the notion of polynomials is more naturally extended in this setting.
\end{remark}

We notice that in Definition~\ref{Multi:process} we claim the existence of a family of $\mathcal F_s$-measurable random variables $(X(s;t))_{0\leq s\leq t<\infty}$, rather than using $X(s)$ as argument on the right-hand side in \eqref{def:multi-process-cond-exp}. As we will see in Example~\ref{mild:sol:exampl:ind:Incr} this allows us to show that mild solutions to certain stochastic partial differential equations as the Ornstein-Uhlenbeck process are polynomial processes with respect to a smart choice of $(X(s;t))_{0\leq s\leq t<\infty}$. To capture these processes also our following definition of independent increment processes is a rather generous one and even for the case $B=\mathbb{R}$ includes processes which are not independent increment processes in the conventional sense. 
We define {\it independent increment processes} on general Banach spaces as:
\begin{definition}\label{ass:process}
The process $(X(t))_{t\geq 0}$ is called an {\tt independent increment process } if
\begin{enumerate}
\item $X(t)$ is strongly $\mathcal F_t$-measurable for any $t\geq 0$,
\item for every $t$ and every $s\leq t$, there exists a decomposition of $X(t)$ into a strongly $\mathcal F_s$-measurable part $X^{\parallel}(s;t)$ and a part $X^{\perp}(s;t)$ that is independent of $\mathcal F_s$ such that $X(t)=X^{\perp}(s;t) + X^{\parallel}(s;t)$,
\item all moments of $\|  X^{\perp}(s;t)  \|$ and $\|  X^{\parallel}(s;t) \|$ are integrable.
\end{enumerate}
\end{definition}
Applebaum~\cite{Apple} defines a L\'evy process on a separable Banach space as a $B$-valued stochastically continuous process $(L(t))_{t\geq 0}$
which is $\mathcal F_t$-adapted, the increments $L(t)-L(s)$ are independent of $\mathcal F_s$ for any $t>s\geq 0$ with distribution only
depending on $t-s$, and having c\`adl\`ag paths.  In view of Definition~\ref{ass:process}, $(L(t))_{t\geq 0}$ will be an independent
increment process with $L^{\perp}(s;t):=L(t)-L(s)$ and $L^{\parallel}(s;t):=L(s)$ as long as all moments of $\|L(t)\|$ are integrable
for any $t\geq 0$.  Property (3) in Definition~\ref{ass:process} follows by the fact that $L(t)-L(s)\stackrel{d}{=}L(t-s)$ by
definition of the L\'evy process. The canonical example of a L\'evy process is the Wiener process. In a separable Banach space,  Fernique's Theorem (see Peszat and Zabczyk~\cite{peszat.zabczyk.07}) ensures the moment condition (3) in Definition~\ref{ass:process} for a Wiener process. Also note that Property (3) in Definition~\ref{ass:process} implies especially that $\E[\|X(t)\|^n]<\infty$ for all $n\in\mathbb N$. We now provide a more interesting example 
of an independent increment process according to Definition~\ref{ass:process}. 

\begin{example}\label{mild:sol:exampl:ind:Incr}
Let $(\mathcal S_t)_{t\geq 0}$ be a $C_0$-semigroup on $B$ and $(W(t))_{t\geq 0}$ a $B$-valued Wiener process. By Fernique's Theorem all moments of $\|W(t)\|$ are finite. 
Consider the stochastic process $(X(t))_{t\geq 0}$ given by
\begin{equation}
\label{stoch-conv}
X(t)=\int_0^t\mathcal S_{t-s}\,dW(s).
\end{equation} 
As $(W(t))_{t\geq 0}$ in particular is square-integrable, it follows that the stochastic convolution $(X(t))_{t\geq 0}$ is a well-defined 
$\mathcal F_t$-adapted process in $B$ (see Applebaum~\cite{Apple} and Peszat and Zabczyk~\cite{peszat.zabczyk.07}). 
Moreover, it is known
(see again Applebaum~\cite{Apple} and Da Prato and Zabczyk~\cite{prato.zabczyk.14}) that $(X(t))_{t\geq 0}$ is a mild solution of the stochastic 
evolution equation
\begin{equation}
\label{inf-OU}
dX(t)=\mathcal A X(t)\,dt+dW(t)
\end{equation} 
where $\mathcal A$ is the (densely defined) generator of $(\mathcal S_t)_{t\geq 0}$. Furthermore, $(X(t))_{t\geq 0 }$ is a symmetric Gaussian process and by Fernique's Theorem all moments of $\| X(t)\|$ are finite.

We decompose $X(t)$ in \eqref{stoch-conv} into $X^{\perp}(s;t) :=\int_s^t\mathcal S_{t-u}dW(u)$ and  $X^{\parallel}(s;t) :=\int_0^s\mathcal S_{t-u}dW(u)$. We find that $X^{\perp}(s;t)$ is independent of $\mathcal F_s$ and
$X^{\parallel}(s;t)$ is $\mathcal F_s$-measurable and thus $(X(t))_{t\geq 0}$ is an independent increment process. Processes of the form (\ref{stoch-conv}) are relevant in financial mathematics. We provide an example from commodity markets later in the application section, Sect. \ref{sec:applic}.
\end{example}

\begin{theorem}
\label{prop:indep-incr-multi}
Suppose that $(X(t))_{t\geq 0}$ is an independent increment process and let $\mathcal{M}_k\in M_k(B)$. Then there exists a family of
$j$th-order monomials $\overline{\mathcal M}_j\in M_j(B)$, $0\leq j\leq k$, such that
\begin{equation}\label{cond:exp:multi}
\mathbb E[\mathcal{M}_k (X(t)) \,|\,\mathcal F_s]=\sum_{j=0}^{k} \overline{\mathcal{M}}_j (X^{\parallel}(s;t)),
\end{equation}
for any $s\leq t$, where the $\overline{\mathcal M}_j$'s depend on $s$ and $t$. In other words, $(X(t))_{t\geq 0}$ is a
multilinear process with respect to $(X^{\parallel}(s;t))_{0\leq s\leq t<\infty}$.
\end{theorem}
\begin{proof}
For $k=0$ the claim is trivial, so assume that $k\geq 1$. Let $\mathcal{L}_k\in L_k(B)$ be such that 
$\mathcal{M}_k (v)=\mathcal{L}_k(v,\dots,v)$. Recall by Definition~\ref{ass:process} that $X(t)=X^{\perp}(s;t) + X^{\parallel}(s;t)$ 
for all $0\leq s\leq t<\infty$, where $X^{\parallel}(s;t)$ is strongly $\mathcal F_s$-measurable and $X^{\perp}(s;t)$ is independent of $\mathcal F_s$. 
Thus we get from multilinearity of $\mathcal L_k$ 
\begin{align*}
\mathcal M_k(X(t))&=\mathcal{L}_k(X(t),\dots,X(t)) \\
&=\mathcal{L}_k(X^{\perp}(s;t) +X^{\parallel}(s;t),\dots,X^{\perp}(s;t)+X^{\parallel}(s;t))\\
&=\mathcal{L}_k(X^{\parallel}(s;t),X^{\perp}(s;t)+X^{\parallel}(s;t),\dots,X^{\perp}(s;t)+X^{\parallel}(s;t)) \\
&\qquad+ \mathcal{L}_k(X^{\perp}(s;t),X^{\perp}(s;t)+X^{\parallel}(s;t),\dots,X^{\perp}(s;t)+X^{\parallel}(s;t)).
\end{align*}
Continuing like this over the remaining $k-1$ arguments one can decompose the above expression into a linear combination of
$2^k$ terms of the form
$$
\mathcal{L}_k(Y_{j,1},\dots,Y_{j,k})
$$
for $1\leq j\leq 2^k$ with $Y_{j,i}\in \{X^{\perp}(s;t),X^{\parallel}(s;t) \}$, and there are exactly $\binom{k}{n}$ terms $j$ for which $\# \{Y_{j,i} \,| Y_{j,i} =X^{\parallel}(s;t) \}=n$. 

Let us look at a particular term where $X^{\parallel}(s;t)$ appears in the first two arguments. Introduce the function 
$\mathcal L_{2,1}: B \times B \rightarrow B$ defined as
\begin{align*}
\mathcal L_{2,1}(y_1,y_2)&= \mathbb E[\mathcal{L}_k(y_1,y_2,X^{\perp}(s;t), \dots,X^{\perp}(s;t)) ] .
\end{align*}
The subscript $(2,1)$ denotes that $\mathcal L_{2,1}$ is the function related to the first term in which $X^{\parallel}(s;t)$ appears twice, where the ordering is irrelevant. 
In view of Proposition~\ref{feezing:lemma} in Appendix~\ref{sec:cond:exp}, let $f(x,y)=\mathcal{L}_k(y,y,x, \dots,x)$, $X=X^{\perp}(s;t)$ and $Y=X^{\parallel}(s;t)$. 
Then, $\sigma (X)=\sigma(X^{\perp}(s;t))$ and $\mathcal F_s$ are independent,  and $Y=X^{\parallel}(s;t)$ is strongly $\mathcal F_s$-measurable. First, we show that $(x,y)\mapsto f(x,y)$ is continuous: Indeed, for $(x,y),(u,v)\in B\times B$, we find by triangle inequality and
$\mathcal L_k\in L_k(B)$ that
\begin{align*}
\|f(u,v)-f(x,y)\|&=\|\mathcal L_k(v,v,u,\ldots,u)-\mathcal L_k(y,y,x,\ldots,x)\| \\
&\leq \|\mathcal L_k(v,v,u,\ldots,u)-\mathcal L_k(y,v,u,\ldots,u)\| \\
&\qquad+\|\mathcal L_k(y,v,u,\ldots,u)-\mathcal L_k(y,y,u,\ldots,u)\| \\
&\qquad+\|\mathcal L_k(y,y,u,\ldots,u)-\mathcal L_k(y,y,x,u,\ldots,u)\| \\
&\qquad \ldots \\
&\qquad+\|\mathcal L_k(y,y,x,\ldots,x,u)-\mathcal L_k(y,y,x,\ldots,x)\| \\
&=\|\mathcal L_k(v-y,v,u,\ldots,u)\|+\|\mathcal L_k(y,v-y,u,\ldots,u)\| \\
&\qquad+\|\mathcal L_k(y,y,u-x,u,\ldots,u)\| \\
&\qquad+\cdots \\
&\qquad+\|\mathcal L_k(y,y,x,\ldots,x,u-x)\| \\
&\leq K\|v-y\|\|v\|\|u\|^{k-2}+K\|y\|\|v-y\|\|u\|^{k-2} \\
&\qquad +K\|y\|^2\|u-x\|\|u\|^{k-3}+\cdots+K\|y\|^2\|u-x\|\|x\|^{k-2} \\
&=K\|u\|^{k-2}(\|v\|+\|y\|)\|v-y\| \\
&\qquad+K\|y\|^2(\sum_{n=0}^{k-3}\|x\|^n\|u\|^{k-3-n})\|u-x\|
\end{align*}
Thus, $(x,y)\mapsto f(x,y)$ is a local Lipschitz continuous map from $B\times B$ into $B$. 

Continuity implies that $f(X,y)$ is strongly measurable (see Lemma~\ref{lemma:cont-preserve-measur} in Appendix~\ref{sec:cond:exp}). Since 
\begin{equation}\label{bound:f}
\| f(X,y) \|\leq K \| y \|^2 \| X \|^{n-2}
\end{equation}
it follows that $f(X,y)$ is Bochner integrable by the finite moments condition on $\|X\|=\|W^{\perp}(s;t)\|$.
Furthermore,
\begin{equation}
\| f(X,\tilde{Y})] \|\leq K \| \tilde{Y} \|^2 \| X \|^{k-2} \leq K \| Y \|^2 \| X \|^{k-2} =:\tilde{Z}
\end{equation}
provided that  $\| \tilde{Y} \| \leq  \| Y \|$ and therefore the bound in (\ref{bound:2}) holds. Again using (\ref{bound:f}) it follows that $\| \E[f(X,y)] \|\leq K \| y \|^2 \E[ \| X \|^{k-2}]$ and 
$$
\| \E[f(X,y)]_{y=\tilde{Y}}  \| \leq K \| \tilde{Y} \|^2 \E[ \| X \|^{k-2}] \leq  K \| Y \|^2 \E[ \| X \|^{k-2}]=:Z 
$$
provided that $\| \tilde{Y} \| \leq \| Y \| $ and the bound in (\ref{bound:1}) holds. 
Moreover, appealing to the fact that $\mathcal L_k\in L_k(B)$ and Bochner's inequality together with the finiteness of all moments of
$\|X\|$, we find by using similar arguments as above that $y\mapsto \E[f(X,y)]$ is locally Lipschitz continuous.
Thus we can apply Proposition~\ref{feezing:lemma} in Appendix~\ref{sec:cond:exp} and conclude that
\begin{align*}
\mathcal L_{2,1}(X^{\parallel}(s;t),X^{\parallel}(s;t))&= \mathbb E[\mathcal{L}_k(y,y,X^{\perp}(s;t), \dots,X^{\perp}(s;t)) ] _{y=X^{\parallel}(s;t)}\\
&=\mathbb E[\mathcal{L}_k(X^{\parallel}(s;t),X^{\parallel}(s;t),X^{\perp}(s;t), \dots,X^{\perp}(s;t))  \,|\,\mathcal F_s].
\end{align*}
By linearity of the expectation operator along with multilinearity of $\mathcal L_k$, the function $\mathcal L_{2,1}(y_1,y_2)$ is bilinear and indeed an element of $L_2(B)$. The same argument applies to the other $\binom{k}{2}-1$ terms 
$\mathcal L_{2,2},\dots , \mathcal L_{2,\binom{k}{2}}$ with $X^{\parallel}(s;t)$ appearing twice. Since the sum of bilinear maps is bilinear we can define the bilinear function:
$$
\widetilde{\mathcal{L}}_2 (y_1,y_2)=\sum_{i=1}^{\binom{k}{2}} \mathcal L_{2,i}(y_1,y_2)
$$
and $\overline{\mathcal{M}}_2(y):=\widetilde{\mathcal{L}}_2 (y,y)\in M_2(B) $. In the same way the other functions $\overline{\mathcal{A}}_j(X^{\parallel}(s;t))$ can be defined for $j\in \{1,3,\dots,k\}$ and the representation \eqref{cond:exp:multi} follows. Thus, the theorem is proved.
\end{proof}

We next observe that elements of $M_k(B)$ share similar characteristics as the monomials on the real line. In fact their $(k+1)$-th Fr\'echet derivative vanishes.

\begin{proposition}\label{prop:frechet:der}
Assume for $k\in\mathbb N$ that $\mathcal M_k\in M_k(B)$ is induced by $\mathcal L_k\in L_k(B)$. 
Then the $n$-th Fr\'echet derivative $D^n\mathcal{M}_k :B\rightarrow L_n (B) $ is given by
\begin{equation}\label{derivative:k:maps}
D^n\mathcal{M}_k (u)(h_1,\dots,h_n)=\sum_{ \substack{ x_i\in \{u, h_1,\dots,h_n \} \\  \# \{i  \,|\, x_i =h_j \}=1 \\ 1\leq j\leq n}} \mathcal{L}_k (x_1,\dots, x_k)
\end{equation}
for $n=1,\dots,k$ and $D^{k+1} \mathcal{M}_k (u)(h_1,\dots,h_{k+1})=0$.
\end{proposition}
\begin{proof}
We have $\mathcal{M}_k(u) = \mathcal{L}_k (u,\dots,u)$. Using the chain rule for the Fr\'echet derivative we then get that 
$$
D\mathcal{M}_k (u) = \sum_{i=1}^k \nabla_i   \mathcal{L}_k (u,\dots,u)\cdot 1.
$$
To calculate $\nabla_i   \mathcal{L}_k (u,\dots,u)$, observe that since $\mathcal{L}_k (u + h_1,\dots,u) - \mathcal{L}_k (u ,\dots,u) - \mathcal{L}_k (h_1 ,\dots,u) =0 $ by multilinearity and therefore by the definition of the Fr\'echet derivative
\begin{equation}
\lim_{\| h_1 \| \rightarrow 0} \frac{ \| \mathcal{L}_k (u + h_1,u,\dots,u) - \mathcal{L}_k (u ,u,\dots,u) - \mathcal{L}_k (h_1 ,u,\dots,u)  \| }{\| h_1 \|}=0.
\end{equation}
Hence, $\nabla_1   \mathcal{L}_k (u,\dots,u)(h_1)=\mathcal{L}_k (h_1 ,u,\dots,u)$, and more generally $\nabla_i   \mathcal{L}_k (u,\dots,u)(h_1)=\mathcal{L}_k (u ,\dots,u,h_1,u, \dots,u)$, where the entry $h_1$ is in the $i$-th coordinate. It follows that, 
$$
D\mathcal{M}_k (u)(h_1)=\sum_{ \substack{ x_i\in \{u, h_1 \} \\  \# \{i  \,|\, x_i =h_1 \}=1}} \mathcal{L}_k (x_1,\dots, x_k).
$$
Clearly  $D\mathcal{M}_k $ maps from $B$ to $L_1(B)=L(B)$.

The claim now follows by induction: assume that \eqref{derivative:k:maps} holds for $n<k$. We pick the term $\mathcal{L}_k (u,\dots,u,h_1,\dots,h_n)$ from the sum in \eqref{derivative:k:maps}. Then 
\begin{align*}
D \mathcal{L}_k (u,\dots,u,h_1,\dots,h_n) (h_{n+1})&=\sum_{i=1}^{k-n} \nabla_i \mathcal{L}_k (u,\dots,u,h_1,\dots,h_n)(h_{n+1})\\
&= \mathcal{L}_k (h_{n+1},u,\dots, u,h_1,\dots,h_n) \\
&\qquad+ \mathcal{L}_k (u,h_{n+1},u,\dots, u,h_1,\dots,h_n) \\
&\qquad + \cdots \\
&\qquad+ \mathcal{L}_k (u,\dots, u,h_{n+1},h_1,\dots,h_n)
\end{align*}
and similarly with all the other terms in \eqref{derivative:k:maps}. We can then compute $D D^n \mathcal{M}_k = D^{n+1}\mathcal{L}_k$,
from which the representation \eqref{derivative:k:maps} follows for $n+1$. 
Directly from this representation one observes that $D^{n+1}\mathcal{M}_k$ can be seen as a map from $B$ to $L_{n+1}(B)$. 
Finally, $D^k\mathcal{M}_k(u)$ is constant and therefore $D^{k+1}\mathcal{M}_k(u)=0$.
\end{proof}
In view of this result, it is fair to call the elements in $M_k(B)$ {\it monomials}, as we have done.

\begin{corollary}
Assume for $k\in\mathbb N$ that $\mathcal M_k\in M_k(B)$ is induced by $\mathcal L_k\in L_k(B)$. Then the $n$-th Fr\'echet derivative $D^n\mathcal{M}_k :B\rightarrow L_n (B)$ is symmetric for any $2\leq n\leq k$, i.e.
\begin{equation}
D^n\mathcal{M}_k (u)(h_1,\dots,h_n)=D^n\mathcal{M}_k (u)(h_{\sigma (1)},\dots,h_{\sigma (n)})
\end{equation}
for any permutation $\sigma \in S_n$, where $S_n$ denotes the set of permutations on $\{ 1, \dots , n \}$.
\end{corollary}
\begin{proof} 
We can rewrite \eqref{derivative:k:maps} as a double sum where we first fix the appearance of the $u$ and then sum over those terms with $u$ in the same coordinate. Again for notational simplicity we look at the specific one with the $u$ fixed to be in the first $k-n$ coordinates and we find that
\begin{align*}
&\sum_{ \substack{ x_i\in \{ h_1,\dots,h_n \} \\  \# \{i  \,|\, x_i =h_j \}=1\\1\leq j \leq n}} \mathcal{L}_k (u,\dots,u,x_1,\dots, x_n) \\
&\qquad=\sum_{\sigma \in S_n} \mathcal{L}_k (u,\dots,u,h_{\sigma (1)},\dots, h_{\sigma (n)})
\end{align*}
and therefore the expression is symmetric. The same argument works for any fixed positions for the $u$'s and there are $\binom{k}{n}$ possible ways to fix them. Therefore $D^n\mathcal{M}_k (u)(h_1,\dots,h_n)$ is the sum of $\binom{k}{n}$ symmetric functions and is therefore itself symmetric. 
\end{proof}

We immediately get the following Corollary which will be important later for polynomials in Banach algebras.
\begin{corollary}\label{cor:der:poly}
Let $B$ be a Banach algebra and $\mathcal L \in L(B)$. Define the $k$-th order monomial $\mathcal M_k(u):=\mathcal L (u^k)$. For $n\leq k$, the $n$-th order Fr\'echet derivative $D^n\mathcal M_k :B\rightarrow L_n(B) $ of $\mathcal M_k$ is given by
\begin{equation}\label{der:poly}
D^n\mathcal M_k(u)(h_1,\dots,h_n)=\mathcal L  \left( \sum_{ \substack{ x_i\in \{u, h_1,\dots,h_n \} \\  \# \{i  \,|\, x_i =h_j \}=1\\1\leq j \leq n}}  x_1\cdot \dots \cdot x_k \right).
\end{equation}
Furthermore if $B$ is commutative, then the expression simplifies to 
\begin{equation}\label{der:poly:commute}
D^n \mathcal M_k(u)(h_1,\dots,h_n)=\frac{k!}{(k-n)!}  \mathcal L  \left( h_1 \cdots h_n u^{k-n}  \right).
\end{equation}
\end{corollary}
\begin{proof}
The monomial $\mathcal M_k$ is induced from the multilinear map $\mathcal{L}_k : (u_1,\dots , u_k) \rightarrow \mathcal L(u_1\cdot \dots \cdot u_k)$. Then \eqref{der:poly} directly follows form Proposition~\ref{prop:frechet:der}.
If $B$ is commutative, then all terms appearing in the sum in \eqref{der:poly} are equal. In fact there are $\binom{k}{n}$ ways to fix the appearance of the $u$ and then $n!$ ways to distribute the $h_1,\dots, h_n$ in the remaining positions. So altogether there are $\frac{k!}{(k-n)!}$ equal terms and \eqref{der:poly:commute} follows.
\end{proof}

\subsection{Multilinear forms}
In this section we shall elaborate a bit on multilinear forms which map into the field $\mathbb{F}$ instead of the Banach space $B$. We first give a precise definition:

\begin{definition}
\label{def:k-maps}
We say that $\mathcal{L}_k : B^k \rightarrow \mathbb{F}$ for $k\in\mathbb N$ is a {\tt $k$-linear form} if it is linear in each argument in the sense that for any
$x_1,x_2,\ldots,x_k,y\in B$ and $a,b\in\mathbb F$ 
\begin{align*}
\mathcal{L}_k (x_1,\dots,x_{j-1},&ax_j +by ,x_{j+1},\dots , x_k) \\ &=a \mathcal{L}_k (x_1,\ldots, x_k) + b \mathcal{L}_k (x_1,\dots,x_{j-1},y ,x_{j+1},\dots, x_k)
\end{align*}
for each $j=1,\ldots,k$. A $k$-linear form $\mathcal L_k$ is {\tt bounded} if there exists a constant $K>0$ such that 
$$
\abs{\mathcal L_k(x_1,\ldots,x_k)}\leq K\|x_1\|\cdots\|x_k\|
$$
for all $x_1,\ldots,x_k\in B$. We denote the space of bounded $k$-linear forms by $L^{\mathbb{F}}_k (B)$. 
\end{definition}
Notice that $L^{\mathbb{F}}_1(B)$ is the dual space of $B$. 
A $k$-linear form $\mathcal{L}_k$ induces a {\it $k$-monomial} $\mathcal{M}_k: B \rightarrow \mathbb{F}$ by
\begin{equation}
\mathcal{M}_k(x):=\mathcal{L}_k (x,\dots,x) .
\end{equation}
If $\mathcal L_k\in L^{\mathbb{F}}_k(B)$, we see that $\abs{\mathcal{M}_k(x)} \leq K\|x\|^k$, and we denote the set of all such $k$-monomials
by $M^{\mathbb{F}}_k(B)$. We use the convention that $M^{\mathbb{F}}_0(B)=\mathbb{F}$. Observe that $M^{\mathbb{F}}_1(B)=L^{\mathbb{F}}_1(B)$.

We show that the multilinearity preserving property implies that also monomials arising from multilinear forms are preserved. We use this result in Section~\ref{sec:applic} for the calculation of conditional moments for Hilbert space valued stochastic processes but the result might be of independent interest.
\begin{proposition}\label{prop:form}
Let $B$ be a Hilbert space with inner product $\langle\cdot,\cdot\rangle$. Let $(X(t))_{t\geq0}$ be a multilinear $B$-valued process with respect to the family of $B$-valued random variables 
$(X(s;t))_{0\leq s\leq t<\infty}$. For every $k$-monomial $\mathcal{M}_k\in M^{\mathbb{F}}_k(B)$, there exist j-monomials  $\overline{\mathcal{M}}_j \in M^{\mathbb{F}}_j(B),j=0,\dots,k$ such that 
\begin{equation*}
\E [\mathcal{M}_k(X(t)) \,|\,\mathcal F_s] =  \sum_{j=0}^k \overline{\mathcal{M}}_j ( X(s;t)).
\end{equation*}
\end{proposition}
\begin{proof}
Let $\mathcal{M}_k\in M^{\mathbb{F}}_k(B)$. Choose $z\in B$ with $\|z\|_B=1$. Define the $k$-monomial $\mathcal{M}^z_k \in M_k(B)$ by $\mathcal{M}^z_k (x) := z\mathcal{M}_k(x)$. Then there exist $\overline{\mathcal{M}}^z_j \in M_j(B)$ for $j=0,\dots, k$ such that
   $$ \E[ \mathcal{M}^z_k (X(t)) | \mathcal F_s ] = \sum_{j=0}^k \overline{\mathcal{M}}^z_j (X(s;t)). $$
 Clearly,  $\overline{\mathcal{M}}_j:=\langle \overline{\mathcal{M}}^z_j(x) , z\rangle$ defines an element in $M^{\mathbb{F}}_j(B)$. Now observe that
  \begin{align*}
    \E[ \mathcal{M}_k(X(t)) | \mathcal F_s] &= \langle  \E[ \mathcal{M}_k^z (X(t)) |\mathcal F_s ] , z \rangle \\
                  &= \langle \sum_{j=0}^k \overline{\mathcal{M}}^z_j (X(s;t)) , z \rangle \\
                  &= \sum_{j=0}^k \overline{\mathcal{M}}_j ( X(s;t)).
  \end{align*}

\end{proof}

\section{Multiplicative maps and polynomials}\label{sec:indep:incr}
We shall now focus on Banach spaces $B$ that are also Banach algebras. We recall that 
when $B$ is a Banach algebra, there is a multiplication operator $\cdot : B \times B \rightarrow B$ defined such that 
$(B, + , \cdot)$ is an associative $\mathbb{F}$-algebra and $\| x\cdot y \| \leq \| x \| \cdot \| y \| $ for any $x,y\in B$.

Suppose that $(X(t))_{t\geq 0}$ is an independent increment process in $B$ (see Definition~\ref{ass:process}), and recall
the decomposition $X(t)=X^{\perp}(s;t) + X^{\parallel}(s;t)$ for any $0\leq s\leq t$.
According to Theorem~\ref{prop:indep-incr-multi}, $(X(t))_{t\geq 0}$ is a multilinear process with respect to $(X^{\parallel}(s;t))_{0\leq s\leq t<\infty}$, i.e., for every $\mathcal{M}_k\in M_k(B), k\in\mathbb N$, there exist $\overline{\mathcal M}_j \in M_j(B), j=0,\dots, k$ 
such that 
\begin{equation}\label{multi:lin:rep}
\E [\mathcal{M}_k(X(t)) \,|\,\mathcal F_s] =  \sum_{j=0}^k \overline{\mathcal{M}}_j ( X^{\parallel} (s;t)),
\end{equation}
for all $0\leq s\leq t$. 
We now look at the particular case of $\mathcal{M}_k(x) =x^k$ 
and address the question under which conditions the induced $j$-order monomials $\overline{\mathcal{M}}_j$ in \eqref{multi:lin:rep} are of polynomial type as well.
\begin{lemma}
\label{lemma:indep-incr-monomial}
Assume that $B$ is a commutative Banach algebra. Then for all $0\leq s\leq t$ 
$$
\mathbb E[X^k(t)\,|\,\mathcal F_s]=q_k(X^{\parallel}(s;t))
$$
with
$$
q_k(x)=\sum_{n=0}^k\binom{k}{n}\mathbb E[(X^{\perp}(s;t))^{k-n}] x^n.
$$
Here, we use the convention that $x^0=1\in\mathbb F$, that is, the term $b_0x^0=b_0\in B$.
\end{lemma}
\begin{proof}
Let $k\in\mathbb N$. By the binomial formula, we find for $s\leq t$
$$
X^k(t)=(X^{\perp}(s;t)+X^{\parallel}(s;t))^k=\sum_{n=0}^k\binom{k}{n}(X^{\perp}(s;t))^{k-n}(X^{\parallel}(s;t))^n.
$$
We get by $\mathcal F_t$-adaptedness of the process and 
Lemma~\ref{lemma:product-measurability} (Appendix~\ref{sec:cond:exp})
that $(X^{\parallel}(s;t))^{n}$ is strongly $\mathcal F_s$-measurable for all $n\leq k$. 
From Proposition~\ref{prop:fact-cond-expect} (Appendix~\ref{sec:cond:exp}),
\begin{align*}
\mathbb E[X^k(t)\,|\,\mathcal{F}_s]&=\sum_{n=0}^k\binom{k}{n}\mathbb E\left[(X^{\perp}(s;t))^{k-n} (X^{\parallel}(s;t))^n \,|\,\mathcal{F}_s\right] \\
&=\sum_{n=0}^k\binom{k}{n}\mathbb E\left[(X^{\perp}(s;t))^{k-n}\,|\,\mathcal F_s\right](X^{\parallel}(s;t))^n \\
&=\sum_{n=0}^k\binom{k}{n}\mathbb E\left[(X^{\perp}(s;t))^{k-n}\right](X^{\parallel}(s;t))^n.
\end{align*}
In the last step, we used that $X^{\perp}(s;t)$ is independent of $\mathcal F_s$ and
in the first step Lemma~\ref{lemma:cond-expect-independence} (Appendix~\ref{sec:cond:exp}). Thus, the lemma follows.
\end{proof}

If $B$ is commutative, define a polynomial $p_k:B\rightarrow B$ of order $k\in\mathbb N$, as
\begin{equation}
p_k(x)=\sum_{n=0}^k b_nx^n\,,
\end{equation} 
where $(b_n)_{n=0}^k\subset B$ and with the convention that 
$x^0=1\in\mathbb F$, i.e., $b_0x^0=b_0\in B$.
If $x\in B$, we find
$$
\|p_k(x)\|\leq\sum_{n=0}^k\|b_n\| \|x\|^n<\infty
$$
by the triangle inequality and Banach algebra norm. We denote the space of polynomials in $B$ of order $k$ by
$\text{Pol}_k(B)$. If $B=\mathbb R$, $\text{Pol}_k(\mathbb R)$ is the space of polynomials on the real line of order $k$. 

From Lemma~\ref{lemma:indep-incr-monomial} it is simple to see that in a commutative Banach algebra 
$$
\E[p_k(X(t)) \,|\,\mathcal F_s]=\sum_{n=0}^kb_nq_n(X^{\parallel}(s;t))=\widetilde{q}_k(X^{\parallel}(s;t))
$$ 
for any $k\in\mathbb N$ and $0\leq s\leq t$, where $\widetilde{q}_k\in\text{Pol}_k(B)$ is given by
$$
\widetilde{q}_{k}(x)=\sum_{n=0}^kb_n\sum_{j=0}^n\binom{n}{j}\E[(X^{\perp}(s;t))^{n-j}] x^j.
$$
This motivates a definition of a {\it polynomial process} in 
a commutative Banach algebra $B$:
\begin{definition}\label{poly:process}
Let $(X(t))_{t\geq 0}$ be a $B$-valued stochastic process where $B$ is a Banach algebra. Furthermore, let 
$(X(s;t))_{0 \leq s \leq t < \infty}$ be a family of $B$-valued random variables, such that $X(s;t)$ is strongly $\mathcal{F}_s$-measurable for
all $0\leq s\leq t$. The process $(X(t))_{t\geq 0}$ is said to be a
{\tt polynomial process} with respect to the family $(X(s;t))_{0 \leq s \leq t < \infty}$ if for all $s\leq t$, $k\in\mathbb N$ and
every polynomial $p_k\in\text{Pol}_k(B)$ there exists a polynomial
$q_m\in\text{Pol}_m(B)$, $m\leq k$ such that
$$
\mathbb E[p_k(X(t))\,|\,\mathcal{F}_s]=q_m(X(s;t)).
$$ 
\end{definition}  
Note that the coefficients of $q_m$ may depend on the times $s$ and $t$. 
We summarize our findings from above in the following Theorem.
\begin{theorem}\label{poly:preserv}
Assume that $B$ is a commutative Banach algebra. Then the independent increment process $(X(t))_{t\geq 0}$ defined in 
Definition~\ref{ass:process} is a polynomial process with respect to $(X^{\parallel}(s;t))_{0\leq s\leq t<\infty}$. 
\end{theorem}

Observe that any $b\in B$ gives rise to a multiplication operator $B\ni x\mapsto bx\in B$, being a
bounded linear operator. Hence, the above definition of polynomials may be viewed as a special case of more
general polynomials with $b_n\in L(B)$, as we define next. 
Define a {\it generalized} polynomial $\mathcal P_k:B\rightarrow B$ of order $k\in\mathbb N$, as
\begin{equation}
\mathcal P_k(x)=\sum_{n=0}^k \mathcal B_n (x^n),
\end{equation} 
where $(\mathcal B_n)_{n=1}^k\subset L(B)$ and $\mathcal B_0\in B$ is a constant reflecting the fact that $x^0=1\in\mathbb F$. If $x\in B$, we find
$$
\|\mathcal P_k(x)\|\leq\sum_{n=0}^k\|\mathcal B_n\|_{\text{op}} \|x\|^n<\infty
$$
by the triangle inequality and Banach algebra norm. Here, $\|\mathcal B_n\|_{\text{op}}$ denotes the operator norm of 
$\mathcal B_n$. We denote the space of generalized polynomials in $B$ of order $k$ by
$\text{gPol}_k(B)$. If $B=\mathbb R$, $\text{gPol}_k(\mathbb R)=\text{Pol}_k(\mathbb R)$ is the space of polynomials on the real line of order $k$. 

We define the following 

\begin{definition}\label{gen:poly:process}
Let $(X(t))_{t\geq 0}$ be a $B$-valued stochastic process where $B$ is a Banach algebra. Furthermore, let 
$(X(s;t))_{0 \leq s \leq t < \infty}$ be a family of $B$-valued random variables, such that $X(s;t)$ is strongly $\mathcal{F}_s$-measurable for
all $0\leq s\leq t$. The process $(X(t))_{t\geq 0}$ is said to be a
{\tt generalized polynomial process} with respect to the family 
$(X(s;t))_{0 \leq s \leq t < \infty}$ if for all $s\leq t$, $k\in\mathbb N$ and
every generalized polynomial $\mathcal P_k\in\text{gPol}_k(B)$ there exists a generalized polynomial
$\mathcal Q_m\in\text{gPol}_m(B)$, $m\leq k$ such that
$$
\mathbb E[\mathcal P_k(X(t))\,|\,\mathcal{F}_s]=\mathcal Q_m(X(s;t)).
$$ 
\end{definition}  
Note that in Definition\ref{poly:process} and \ref{gen:poly:process} we do not assume that $B$ is commutative. The name {\it generalized polynomial process} is justified by the following Proposition which states that every polynomial 
 process is also a generalized polynomial process. 

\begin{proposition}\label{prop:genpoly:poly}
Assume $B$ is a Banach algebra, and $(X(t))_{t\geq 0}$ is a polynomial process in $B$ with respect to the family 
$(X(s;t))_{0 \leq s \leq t < \infty}$. Then $(X(t))_{t\geq 0}$ is also a generalized polynomial process in $B$ with respect to the family $(X(s;t))_{0 \leq s \leq t < \infty}$.
\end{proposition}
\begin{proof}
For $k\in\mathbb N$, let $\mathcal P_k\in\text{gPol}_m(B)$ with \begin{equation}
\mathcal P_k(x)=\sum_{n=0}^k \mathcal B_n(x^n)\,.
\end{equation} 
Because $(X(t))_{t\geq 0}$ is a polynomial process with respect to $(X(s;t))_{0 \leq s \leq t < \infty}$, it follows that 
for each $0\leq n\leq k$ there exists
a $q_{m(n)}\in\text{Pol}_{m(n)}(B)$ for $m(n)\leq k$ such that for any $s\leq t$
\begin{equation}
\E [(X(t))^n\,|\,\mathcal{F}_s] = q_{m(n)}(X(s;t)) =\sum_{j=0}^{m(n)} b_{n,j}   (X(s;t))^j
\end{equation}
where $(b_{n,j})_{j=0,\dots,m(n)}\subset B$.
It follows using Lemma~\ref{exchange:int:operator} that 
\begin{align*}
\E [\mathcal B_n ( (X(t))^n ) \,|\,\mathcal{F}_s] &= \mathcal B_n (\E [ ( (X(t))^n ) \,|\,\mathcal{F}_s] ) \\
&= \mathcal B_n (\sum_{j=0}^{m(n)} b_{n,j} (X(s;t))^j)\\
&=  \sum_{j=0}^{m(n)} \mathcal B_n (b_{n,j}   (X(s;t))^j)\\
&= \sum_{j=0}^{m(n)} \widetilde{\mathcal B}_{n,j}   (X(s;t))^j
\end{align*}
with $\widetilde{\mathcal B}_{n,j} \in L(B)$ being defined by $B\ni x \mapsto \mathcal B_n (b_{n,j} x)\in B$. Define the 
generalized polynomial $\mathcal Q_{\widehat{m}}\in\text{gPol}_{\widehat{m}}(B)$ by $\mathcal Q_{\widehat{m}} (x) = \sum_{n=0}^k  \sum_{j=0}^{m(n)}
 \widetilde{\mathcal B}_{n,j} x^j$ with $\widehat{m}:=\max_{0\leq n\leq k} m(n)\leq k$. The result follows.
\end{proof}

To see that the opposite does not hold in general we return to Example~\ref{mild:sol:exampl:ind:Incr} and now assume that $B$ is a separable Hilbert space with a commutative algebra defined. 

As before we decompose $X(t)$ as defined in \eqref{stoch-conv} into $X^{\perp}(s;t) :=\int_s^t\mathcal S_{t-u}dW(u)$ and  $X^{\parallel}(s;t) :=\int_0^s\mathcal S_{t-u}dW(u)$ where $X^{\perp}(s;t)$ is independent of $\mathcal F_s$ and
$X^{\parallel}(s;t)$ is $\mathcal F_s$-measurable. Hence, $(X(t))_{t\geq 0}$ is an independent increment process
in $B$ and by Theorem~\ref{poly:preserv}, it holds for any $k\in\mathbb N$ and $s\leq t$
\begin{align*}
\E[X^k(t)\,|\,\mathcal F_s]&
=\sum_{n=0}^k\binom{k}{n}\E[(X^{\perp}(s;t))^{k-n}]
(X^{\parallel}(s;t))^n.
\end{align*}
Hence, as expected, $(X(t))_{t\geq 0}$ is a polynomial process with respect to the family $(X^{\parallel}(s;t))_{0\leq s\leq t<\infty}$. 
Let us analyse the situation a few steps further: From the semigroup property of $(\mathcal S_t)_{t\geq 0}$, we find that 
$$
X^{\parallel}(s;t)=\mathcal S_{t-s}\int_0^s\mathcal S_{s-u}dW(u)=\mathcal S_{t-s}X(s).
$$
Thus,
\begin{equation*}
\E[X^k(t)\,|\,\mathcal F_s]=\sum_{n=0}^k\binom{k}{n}\E[(X^{\perp}(s;t))^{k-n}]
(\mathcal S_{t-s}X(s))^n. 
\end{equation*}
Now, assume $(\mathcal S_t)_{t\geq 0}$ is an algebra homomorphism so that $\mathcal S_t(x\cdot y)=(\mathcal S_t x)\cdot(\mathcal S_t y)$ for any $x,y\in B$. Then $(\mathcal S_t x)^n=\mathcal S_t x^n$ for all $n\in\mathbb N$ and we find
\begin{equation}
\label{example-OU}
\E[X^k(t)\,|\,\mathcal F_s]=\sum_{n=0}^k\binom{k}{n}\E[(X^{\perp}(s;t))^{k-n}]\mathcal S_{t-s} X^n(s). 
\end{equation}
This shows that $(X(t))_{t\geq 0}$ is a {\it generalized} polynomial process with respect to the family 
$(X (s))_{0\leq s \leq t <\infty}$. This is in line with the definition of finite-dimensional polynomial processes (see 
Cuchiero et al.~\cite{CKT} and Filipovi\'c and Larsson \cite{LF}), and the fact that we can establish the generalized polynomial preserving property of the process
with respect to itself is significantly stronger and more applicable than merely in terms of some family of 
$\mathcal F_s$-measurable random variables.  
On the other hand, $(X(t))_{t\geq 0}$ is in general {\it not} a polynomial process with respect to $(X (s))_{0\leq s \leq t <\infty}$, as the coefficients 
on the right hand side of \eqref{example-OU} are elements in $L(B)$. In fact, when $B$ is a function space with multiplication defined point-wise and $\mathcal{S}_t$ is the shift operator (such an example is considered in more detail in Section~\ref{subsec:commodity}), then for instance $\E [X(t) \,|\,\mathcal F_s ] = \mathcal{S}_{t-s} X(s)= X(s) (t-s+ \cdot)$. However, if $(X(t))_{t\geq 0}$ is polynomial, then for every $y := t - s$ there must exist $a_y,b_y \in B$ such that $\E [X(t) \,|\,\mathcal F_s ] =X(s) (t-s+ \cdot) =X(s) (y+ \cdot) = a_y X(s)(\cdot) + b_y$. Since $\mathcal{S}_{t-s}$ is linear we may conclude that $b_y=0$. Then evaluating at $0$ leads in particular to $X(s)(y) = a_y (0)X(s)(0)$. This implies that $X(s)(\cdot)$ is measurable with respect to the sigma algebra $\sigma(X(s)(0))$. However, this cannot be the case unless the driving Wiener process $(W(t))_{t\geq 0}$ is one-dimensional. This provides us with an example of a process that is generalized polynomial but not polynomial.

It is worth emphasising that the above analysis shows that, in general, stochastic convolutions as in \eqref{stoch-conv} are
polynomial processes with respect to $(\mathcal S_{t-s}X(s))_{0\leq s\leq t<\infty}$. Indeed, they are generalized polynomial
processes with respect to $(X(s))_{0\leq s\leq t<\infty}$ when the semigroup is a homomorphism, but not necessarily polynomial with
respect to the same family. Hence, the extension of Ornstein-Uhlenbeck processes to infinite dimensions as in \eqref{inf-OU} 
is not straightforwardly preserving the natural polynomial property from the finite-dimensional case.

The class of processes defined in \eqref{stoch-conv} is of interest from the application point of view. Stochastic evolution equations
like the Ornstein-Uhlenbeck process in \eqref{inf-OU} appear in many applications, for example as the heat equation in random media
(see e.g. Walsh~\cite{W}) or as the dynamics of forward prices in finance and commodity markets (see e.g. Benth and Kr\"uhner~\cite{BK}). We return to the latter in the next Section.

\subsection{Counterexample: non-commutative case}
Consider the case when $B$ is a non-commutative Banach algebra. Then the binomial
formula used in the proof of Lemma~\ref{lemma:indep-incr-monomial} and later above does not hold. For example,
if $(X(t))_{t\geq 0}$ is an independent increment process in $B$, we find for $t\geq s$ that
\begin{align*}
\E[X^3(t)\,|\,\mathcal F_s]&=\E[(X^{\perp}(s;t))^3\,|\,\mathcal F_s]+
\E[(X^{\perp}(s;t))^2 X^{\parallel}(s;t) \,|\,\mathcal F_s] \\
&\quad+\E[(X^{\perp}(s;t)) X^{\parallel}(s;t) (X^{\perp}(s;t))\,|\,\mathcal F_s]+\E[(X^{\perp}(s;t)) (X^{\parallel}(s;t))^2 \,|\,\mathcal F_s] \\
&\quad+\E[X^{\parallel}(s;t) (X^{\perp}(s;t))^2\,|\,\mathcal F_s]+\E[X^{\parallel}(s;t)(X^{\perp}(s;t))X^{\parallel}(s;t) \,|\,\mathcal F_s] \\
&\quad+\E[(X^{\parallel}(s;t))^2(X^{\perp}(s;t))\,|\,\mathcal F_s]+\E[ (X^{\parallel}(s;t))^3 \,|\,\mathcal F_s] \\
&=\E[(X^{\perp}(s;t))^3]+E[(X^{\perp}(s;t))^2]X^{\parallel}(s;t) \\
&\quad+\E[(X^{\perp}(s;t))X^{\parallel}(s;t)(X^{\perp}(s;t))\,|\,\mathcal F_s] \\
&\quad+\E[(X^{\perp}(s;t))] (X^{\parallel}(s;t))^2+X^{\parallel}(s;t)\E[(X^{\perp}(s;t))^2] \\
&\quad+(X^{\parallel}(s;t))^2\E[(X^{\perp}(s;t))]+(X^{\parallel}(s;t))^3
\end{align*}
after appealing to independence and measurability using Lemmas~\ref{lemma:cond-expect-independence} 
and \ref{prop:fact-cond-expect} (Appendix~\ref{sec:cond:exp}). It is not immediately clear how to deal with the term involving the
conditional expectation of $X^{\perp}(s;t)X^{\parallel}(s;t)X^{\perp}(s;t)$, and thus how to express
$\E[X^3(t)\,|\,\mathcal F_s]$ as a polynomial in $X^{\parallel}(s;t)$. 

Using Proposition~\ref{feezing:lemma} (Appendix~\ref{sec:cond:exp}), we know that 
$$
\E[(X^{\perp}(s;t))X^{\parallel}(s;t)(X^{\perp}(s;t))\,|\,\mathcal F_s]=\E [(X^{\perp}(s;t)) y (X^{\perp}(s;t))  ]_{y=X^{\parallel}(s;t)}
$$ 
and observe that $B\ni y\mapsto\E [(X^{\perp}(s;t)) y (X^{\perp}(s;t))]\in B$ is a linear function. Furthermore, it is bounded as 
\begin{align*}
 \| \E [(X^{\perp}(s;t)) y (X^{\perp}(s;t))  ] \| & \leq \E [ \| (X^{\perp}(s;t)) y (X^{\perp}(s;t))  \| ] \\ 
& \leq \E [ \| (X^{\perp}(s;t)) \| \| y \| \| (X^{\perp}(s;t))  \| ] \\
&= \| y \| \E [ \| (X^{\perp}(s;t)) \|^2   ].
\end{align*}
Altogether this means that 
$$ 
\E [X^{\perp}(s;t)X^{\parallel}(s;t) X^{\perp}(s;t) \,|\,\mathcal{F}_s ]=\mathcal L (X^{\parallel}(s;t))
$$
for the bounded operator $\mathcal L \in L(B)$ defined by $y\mapsto \E [X^{\perp}(s;t)y X^{\perp}(s;t) ] $. To show that in fact $\mathcal L $ is not a (left)-multiplication operator we look at the vector space 
$\mathbb{R}^{2\times 2}$ of $2 \times 2$-matrices equipped with a sub-multiplicative matrix norm and the usual
matrix product. This space is well-known to be a non-commutative Banach algebra. 
Let now $(L_{ij}(t))_{t\geq 0}$ for $i,j=1,2$ be 4 independent copies of the real-valued 
L\'evy processes $(L(t))_{t\geq 0}$ with finite moments of all orders. Then
$$
X(t)=\begin{pmatrix}
L_{11}(t) & L_{12}(t) \\
L_{21}(t) & L_{22}(t)
\end{pmatrix},
$$
defines an independent increment process in the space of $2\times 2$-matrices. It follows that
$$
X^{\perp}(s;t)
=\begin{pmatrix}
\Delta_{s,t} L_{11} & \Delta_{s,t} L_{12} \\
\Delta_{s,t} L_{21} & \Delta_{s,t}L_{22}
\end{pmatrix}
$$
where $\Delta_{s,t}L_{ij}(t)=L_{ij}(t)-L_{ij}(s)$ for $s\leq t$. 
Choose now
$$
h:= \begin{pmatrix} 
0 & 1 \\
0 & 0 
\end{pmatrix}.
$$
We obtain that 
$$
\mathcal L (h)=\begin{pmatrix}
\E [\Delta_{s,t} L]^2 & \E [\Delta_{s,t} L]^2 \\
\E[ (\Delta_{s,t} L)^2 ] & \E [\Delta_{s,t} L]^2
\end{pmatrix}=\begin{pmatrix}
0 & 0 \\
1 & 0
\end{pmatrix}=:g
$$
if $L$ is such that $\E[L(t)] = 0$ and $s < t$ such that $\E [(\Delta_{s,t}L)^2]=1$, where $ \Delta_{s,t}L = L(t)-L(s) $. However, one easily verifies that no matrix $a \in \mathbb{R}^{2\times 2}$ exists with $ah=g$.

This shows that the independent increment process $(X(t))_{t\geq 0}$ is
{\it not} a polynomial process in a non-commutative Banach algebra $B$. This is very different from the commutative case, where we recall from Theorem~\ref{poly:preserv} that independent increment processes are polynomial processes.

Motivated by the above derivation, we may ask the question whether $(X(t))_{t\geq 0}$ is a {\it generalized} polynomial process. However, this is also not the case as can be seen by looking at $\E[X^5(t)\,|\,\mathcal F_s]$: Similar calculation as above yields one term of the form
\begin{align*}
 &  \E [X^{\perp}(s;t) X^{\parallel}(s;t) X^{\perp}(s;t)   X^{\parallel}(s;t) X^{\perp}(s;t) \,|\,\mathcal{F}_s ] \\
 &\qquad\qquad= \E [X^{\perp}(s;t) y X^{\perp}(s;t)   y X^{\perp}(s;t)  ]_{y=X^{\parallel}(s;t)}
\end{align*}
and the question is whether this expression can be written as $\mathcal L_2( (X^{\parallel}(s;t))^2 ) + \mathcal L_1( (X^{\parallel}(s;t)) ) + b$ for some $\mathcal L_1 , \mathcal L_2 \in L(B), b\in B$. Let us assume that this is indeed the case, that is, 
$$
 f(y):=\E [X^{\perp}(s;t) y X^{\perp}(s;t) y X^{\perp}(s;t)]=\mathcal L_2 (y^2) + \mathcal L_1 (y) +b.
$$
By Proposition~\ref{prop:frechet:der} we know that 
\begin{align*}
D^2 f(y)(h_1,h_2) &= \E [X^{\perp}(s;t) h_1 X^{\perp}(s;t) h_2 X^{\perp}(s;t)] \\
&\qquad+ \E [X^{\perp}(s;t) h_2 X^{\perp}(s;t) h_1 X^{\perp}(s;t)]
\end{align*}
and by Corollary~\ref{cor:der:poly} that
\begin{equation*}
D^2 ( \mathcal L_2 (y^2) ) (h_1,h_2) = \mathcal L_2 (h_1 h_2) + \mathcal L_2 (h_2 h_1).
\end{equation*}
and $D^2 ( \mathcal L_1  +b) =0$.
If  $f(y)= \mathcal L_2 (y^2) + \mathcal L_1 (y) +b$ then of course also their derivatives agree and
$$D^2 f(y)(h_1,h_2) = D^2 (\mathcal L_2 (y^2)) (h_1,h_2)$$
 for every $h_1,h_2 \in B$. 
 
 We now choose 
$$
h_1:= \begin{pmatrix} 
1 & 0 \\
0 & 0 
\end{pmatrix}
\quad 
h_2:= \begin{pmatrix} 
0 & 0 \\
0 & 1 
\end{pmatrix}
$$
and first observe that $h_1\cdot h_2 = h_2 \cdot h_1= {\bf 0}_2$ with ${\bf 0}_2$ being the $2\times 2$-matrix of zeros, then 
$D^2 (\mathcal L_2 (y^2)) (h_1,h_2)=\mathcal L_2 ({\bf 0}_2 ) + \mathcal L_2 ({\bf 0}_2 ) ={\bf 0}_2$ independent of the specification of 
$(X(t))_{t\geq 0}$.

 With the same choice for $(X(t))_{t\geq 0}$ as above we derive,
\begin{eqnarray}\nonumber 
&&D^2 f(y)(h_1,h_2)\\\nonumber 
&=&\begin{pmatrix} 2\E[\Delta_{s,t}L]^3 & \E[\Delta_{s,t}L]^3  + \E[(\Delta_{s,t}L)^2] \E[\Delta_{s,t}L] \\ \E[\Delta_{s,t}L]^3+ \E[(\Delta_{s,t}L)^2] \E[\Delta_{s,t}L] & 2\E[\Delta_{s,t}L]^3\end{pmatrix} \neq {\bf 0}_2
\end{eqnarray}
whenever $\E[L(t)]\neq 0$. Choosing now a real-valued L\'evy process with mean unequal to zero yields 
a contradiction to $D^2 (\mathcal L_2 (y^2)) (\bar{h}_1,\bar{h}_2)={\bf 0}_2$. So, in general independent increment processes fail to be even generalized polynomial processes in a non-commutative Banach algebra. This shows that even for general Banach algebras one must introduce monomials in $M_k(B)$ as the structure preserving class to extend the notion of "polynomial processes" 
to infinite dimensions, and not merely polynomials nor generalized polynomials.

\section{Applications}\label{sec:applic} In this section we want to elaborate a bit more on some of the possible applications. 
\subsection{Calculation of moments}
For multilinear processes in Hilbert space we can compute conditional moments of the norm of the process. To this end, suppose $B$ is a Hilbert space with inner product
$\langle\cdot,\cdot\rangle$.
Define for $k\in\mathbb N$,
$$
\mathcal L_{2k}(x_1,y_1,\ldots,x_k,y_k):=\langle x_1,y_1\rangle\cdots\langle  x_k,y_k\rangle.
$$
$\mathcal L_{2k}$ is a multilinear form, which is obviously bounded. We have
$$
\mathcal M_{2k}(x):=\|x\|^{2k}
$$ 
for any $x\in B$. Thus, if $X$ is a multilinear process with respect to the family of $B$-valued random variables 
$(X(s;t))_{0\leq s\leq t<\infty}$, then by Proposition~\ref{prop:form},
$$
\E[\|X(t)\|^{2k}  \,|\,\mathcal F_s] ]=\E[\mathcal M_{2k}(X(t)) \,|\,\mathcal F_s]]=\sum_{j=1}^{2k}\overline{\mathcal M}_j(X(s;t)),
$$
for some $j$th-order monomials $\overline{\mathcal M}_j:B\rightarrow \mathbb{K}$, $j=1,\ldots,2k$. 
So, we can compute even moments of the norm of $X$ in terms of multilinear forms operating on $X(s;t)$ of order at most $2k$, where
$X(s;t)$ is $\mathcal F_s$-measurable.

For the odd moments, we note that for $k=0,1,2,\ldots,$ it obviously holds that $2k+1=\alpha(k)\times (2k+2)$ for $\alpha(k)=(2k+1)/(2k+2)\in (1/2,1)$. One has that (see Applebaum~\cite{Apple-book}, page 80)
$$
u^{\alpha(k)}=\frac{\alpha(k)}{\Gamma(1-\alpha(k))}\int_0^{\infty}(1-e^{-u x})x^{-\alpha(k)-1}dx
$$
Thus, we find the representation
$$
\E[\|X(t)\|^{2k+1}  \,|\,\mathcal F_s ]=\frac{\alpha(k)}{\Gamma(1-\alpha(k))}\int_0^{\infty}(1-\E[\exp(-x\|X(t)\|^{2k+2})\,|\,\mathcal F_s])x^{-1-\alpha(k)}dx
$$
Doing a series representation of the exponential function inside the integral on the right hand side, we find that
$$
\E[\|X(t)\|^{2k+1}\,|\,\mathcal F_s]=\frac{\alpha(k)}{\Gamma(1-\alpha(k))}\int_0^{\infty}\sum_{\ell=1}^{\infty}\frac{(-1)^\ell}{\ell!}x^{-\alpha(k)-1+\ell}\E[\|X(t)\|^{2\ell(k+1)}\,|\,\mathcal F_s]dx
$$
Thus, we can use the multilinear property of a process $(X(t))_{t\geq 0}$ to compute an integral of an infinite series of 
even moments to find any odd moment of $\| X(t)\|$. Whether above formula is useful depends of course on the particular situation and possible closed form alternatives.

We remark that the Ornstein-Uhlenbeck process driven by a Wiener process considered in Example~\ref{mild:sol:exampl:ind:Incr} is Gaussian and if $B$ is a separable Hilbert space we may calculate $\Vert X(t)\Vert^2=\sum_{i=1}^{\infty}\langle X,e_i\rangle^2$, where $\{ e_i \}_{i\in \mathbb{N}}$ is an orthonormal basis for $B$. Hence $\Vert X(t)\Vert^{2k}=\sum_{i_1,\dots,i_k}\langle X,e_{i_1}\rangle^2 ... \langle X,e_{i_k}\rangle^2$. By the Isserlis-Wick theorem the $2k$-mixed moments $\E[\langle X,e_{i_1}\rangle^2 ... \langle X,e_{i_k}\rangle^2]$ may be calculated explicitly from the covariance operator of $X(t)$. As above one may then calculate the odd moments from the even moments. Even in the Gaussian case however, it is not clear how to calculate conditional moments and above proposed procedure may be used.

In case the space $B$ is a Banach algebra, we may calculate $\E[ X(t)^{k}  \,| \,\mathcal F_s]$ by Lemma~\ref{lemma:indep-incr-monomial} for (generalized) polynomial processes. To the best of our knowledge calculations of moments of this form have not been investigated yet.
 

\subsection{Applications to commodity markets}\label{subsec:commodity}
A forward contract is a financial arrangement where the seller promises to deliver an underlying commodity (like for example oil, coffee,
aluminium or power) at an agreed price at some future time point. Entering such a contract at time $t\geq 0$ where delivery takes place 
at time $t+x, x\geq 0$ in the future, we denote the agreed {\it forward price} by $f(t,x)$. It is known 
(see Benth and Kr\"uhner~\cite{BK}) that $t\mapsto f(t,\cdot)$ can be interpreted as a stochastic process with values in
some Hilbert space of continuous functions on $\mathbb R_+$, solving (mildly) the stochastic partial differential equation
\eqref{inf-OU} with  $\mathcal A=\partial/\partial x$. This model is a special class of a more general stochastic partial differential equation dynamics, belonging to the Heath-Jarrow-Morton modelling paradigm (see e.g. Filipovi\'c \cite{F}, Geman~\cite{Geman}, Carmona and
Tehranchi~\cite{CT} for more on this, including the case of forward rates in fixed-income markets). 

Following Benth and Kr\"uhner~\cite{BK}, a natural state space of the forward price curves is the Filipovi\'c space
(see Filipovi\'c~\cite{F}). The Filipovi\'c space $H_w$ is defined for an increasing, continuous function $w: \mathbb{R}_+ \rightarrow [1, \infty ) $ with $w(0)=1$ to be the set of functions
\begin{equation}
H_w:= \left\{ g \in AC(\mathbb{R}_+,\mathbb{R}): \int_0^\infty w(x) g' (x)^2 dx < \infty \right\},
\end{equation}
where $AC(\mathbb{R}_+,\mathbb{R})$ denotes the set of absolutely continuous functions from $\mathbb{R}_+$ to $\mathbb{R}$. The scalar product $\langle g_1,g_2 \rangle := g_1 (0) g_2 (0) + \int_0^\infty w(x) g_1' (x) g_2' (x) dx $ for $g_1, g_2 \in H_w$ makes $H_w$ a separable Hilbert space with norm $\norm{g}_w^2 = \abs{\langle g,g \rangle }$. As already observed in 
Benth and Kr\"uhner~\cite{BK} assuming $w^{-1}\in L^1(\mathbb R_+)$, the pointwise multiplication defines an algebra on 
$H_w$ and with the new norm $|\cdot|_{w,c}:=c|\cdot|_w$, where 
$c=\sqrt{1+8(1+\int_0^{\infty}w^{-1}(x)\,dx)}$ the space $H_w$ is actually a commutative Banach algebra.  

On the commutative Banach algebra $B=H_w$, we have that the shift operator $\mathcal S_t g:=g(\cdot+t)$ defines a 
$C_0$-semigroup being a homomorphism. Moreover, the generator of $(\mathcal S_t)_{t\geq 0}$ is the derivative
operator $\partial/\partial x$. Thus, in light of the discussion in Section~\ref{sec:indep:incr}, the forward curve dynamics
$(f(t,\cdot))_{t\geq 0}$ is given by the stochastic convolution process \eqref{stoch-conv}, and recalling 
\eqref{example-OU}, will become a 
generalized polynomial process on $H_w$ with respect to $(f(s,\cdot))_{0\leq s\leq t<\infty}$. In representation \eqref{example-OU},
we will have  $X^{\perp}(s;t) :=\int_s^t\mathcal S_{t-u}dW(u)$ with $\mathcal S_t$ being the shift operator. In addition it is also a multilinear process with the same decomposition.

Let us give an application where the generalized polynomial property comes in handy.  In commodity markets, options on forwards are popular risk management products. Let us consider a general payoff given by a measurable function $h:\mathbb R\rightarrow \mathbb R$ on the forward with delivery time $x$. At time $t$, the holder 
can exercise the option yielding a payment 
$$
h(f(t,x)).
$$ 
The most prominent example is $h(z)=\max(z-K,0)$ for a standard call option. 
Let $\delta_x$ denote the evaluation map at $x\geq 0$. It is shown in Filipovi\'c~\cite{D} that $\delta_x$ is a bounded linear functional on $H_w$. Thus,
$$
h(f(t,x))=h(\delta_{x} f(t)),
$$
and the price of the option at time $s\leq t$ is given by
\begin{equation}
P(s,t)=\E\left[h\left(\delta_{x}f(t)\right)\,|\,\mathcal F_s\right]
\end{equation}
assuming zero risk-free interest rate (see Benth et al.~\cite{BSBK}). Assume now that there exists a polynomial representation of
$h$, 
\begin{equation}
h(z)=\sum_{i=0}^{\infty}h_{i} z^i.
\end{equation}
For continuous functions such approximation is guaranteed by the Weierstrass Approximation Theorem by using Bernstein polynomials. Let further
\begin{equation}
P(s,t)=\sum_{i=0}^{\infty}h_{i}\E\left[(\delta_{x}f(t))^{i}\,|\,\mathcal F_s\right].
\end{equation}

As $\delta_x$ is a linear functional on $H_w$, using that $\delta_x(g\cdot h)=\delta_x(g)\delta_x(h)$ for any $g,h\in H_w$, we can apply a modified version of 
Lemma~\ref{exchange:int:operator} to show that 
$$
\delta_{x}\E[f(t)^{i}\,|\,\mathcal F_s]=\E[\delta_{x}(f(t)^{i})\,|\,\mathcal F_s]=\E[(\delta_{x}f(t))^{i}\,|\,\mathcal F_s]
$$
It follows
from \eqref{example-OU} that
\begin{align*}
P(s,t)&=\sum_{i=0}^{\infty}h_{i}\delta_{x}\E[f(t)^{i}\,|\,\mathcal F_s]\\
&=\sum_{i=0}^{\infty}h_{i}\delta_{x}\E[(X^{\perp}(s;t)+ X^{\parallel}(s;t)  )^{i}\,|\,\mathcal F_s] \\
&=\sum_{i=0}^{\infty}h_{i} \left(\sum_{k=0}^{i}\binom{i}{k}\delta_{x}(\E\left[(X^{\perp}(s;t))^{i-k}\right])\delta_{x+t-s}f^i(s)\right)
\end{align*}
where we have used that $\delta_x(g\cdot h)=\delta_x(g)\delta_x(h)$ for any $g,h\in H_w$ and $\delta_x\mathcal S_t=\delta_{x+t}$ for every $x,t\geq 0$. Note that the quantities $\delta_{x+t-s}f^i(s)= f^i(s,x+t-s),1\leq i < \infty$ can be read off from the observed forward curve $f(s)$ at time $s$. Choosing a sufficiently large cutoff $n$, an approximation for the price of the option is given by
$$
P(s,t)\approx \sum_{i=0}^{n}h_{i} \left(\sum_{k=0}^{i}\binom{i}{k}\delta_{x}(\E\left[(X^{\perp}(s;t))^{i-k}\right])\right)f^i(s,x+t-s).
$$
In some cases $\E\left[(X^{\perp}(s;t))^{i-k}\right]$ may be calculated explicitly as for example if $(X(t))_{t\geq 0}$ is the Ornstein-Uhlenbeck process studied in Example~\ref {mild:sol:exampl:ind:Incr}. In other cases the calculation of $\E\left[(X^{\perp}(s;t))^{i-k}\right]$ requires numerical techniques. However, these quantities need to be calculated only once in order to evaluate a whole trading book with options maturing at the same time.

A similar representation can be found if the option holder can exercise the option to receive a payment 
$$
h(f(t,x_1),f(t,x_2),\ldots,f(t,x_n))
$$ 
with delivery times $0\leq x_1<x_2<\cdots<x_n$ and $h:\mathbb R^n\rightarrow \mathbb R$ a measurable function. This generalization covers the important example of a spread option on two forwards with different delivery times (calendar spread option) in which case $n=2$ and $h(z_1,z_2)=\max(z_1-z_2,0)$. For these more complicated options one has to use the fact that $(f(t,\cdot))_{t\geq 0}$ is a multilinear process with respect to $(X^{\parallel}(s;t))_{0\leq s\leq t<\infty}$. 

\begin{remark}
Recall the definition of the process $(X(t))_{t\geq 0}$, which was the basis for the previous example, based on the stochastic partial differential equation (\ref{inf-OU}). To allow for a wider range of applications, it would be interesting to analyse processes $(X(t))_{t\geq 0}$ given as the mild solution of the stochastic partial differential equation
\begin{equation}
dX(t)=\mathcal A X(t) + a(t,X(t))\,dt+b(t,X(t)) dW(t)
\end{equation} 
where $\mathcal A$ is again a (densely defined) generator of a $C_0$-semigroup $(\mathcal S_t)_{t\geq 0}$ and $a: \mathbb{R}_+ \times B \rightarrow B, b : \mathbb{R}_+ \times B \rightarrow L(\widetilde{B},B)$. Here $\widetilde{B}$ is the noise space (possibly equal to $B$) and $(W(t))_{t\geq 0}$ is a Wiener or L\'evy process defined on $\widetilde{B}$. Such equation has under appropriate conditions a mild solution (see for instance Theorem 4.5. (1) in Tappe \cite{T2012}) given by
\begin{equation}
X(t)=\mathcal S_t X(0) +\int_0^t\mathcal S_{t-s} a(s,X(s)) \,ds  +  \int_0^t\mathcal S_{t-s} b(s,X(s)) \,dW(s).
\end{equation}
An interesting question for future research is then to find out if there exists a family $(X(s;t))_{0 \leq s \leq t < \infty}$ of $B$-valued random variables, such that $X(s;t)$ is strongly $\mathcal{F}_s$-measurable and such that $(X(t))_{t\geq 0}$ is a multilinear or (generalized) polynomial process with respect to the family $(X(s;t))_{0 \leq s \leq t < \infty}$.
\end{remark}
\subsection{Some other choices of Banach spaces}

A canonical example of a separable Banach space is the space $C([0,1])$ of real-valued continuous functions 
$f:[0,1]\rightarrow\mathbb R$ equipped with pointwise product and
uniform norm $|f|_{\infty}:=\sup_{x\in[0,1]}|f(x)|$. This is also a commutative Banach algebra, and we notice that it is
the path space of Brownian motion. 

Another classical separable Banach space is $L^p(\mathbb R^d)$, the space
of $p$-integrable functions on $\mathbb R^d$ for $p,d\in\mathbb N$. As is well-known, $L^2(\mathbb R^d)$ is
a Hilbert space and also possible state-space for Gaussian random fields. One can define a multiplication for 
$f,g \in L^1 (\mathbb{R}^d)$ by the convolution product, i.e. 
\begin{equation}
f \ast g (x) =\int_{\mathbb R^d} f (y-x) g (y) dy.
\end{equation}
This turns $L^1(\mathbb R^d)$ into a commutative Banach algebra. A possible application could be stochastic processes
taking values in $L^1(\mathbb R)$ being probability densities, e.g. being non-negative integrable functions with unit
mass.

Another classical Banach algebra is the space of bounded linear operators $B=L(C)$ on the Banach space $C$, equipped with the 
operator norm. The space $B$ forms a 
non-commutative Banach algebra under the standard operator product. If $C$ is separable Hilbert 
space, one can consider the subspace of Hilbert-Schmidt operators $L_{\text{HS}}(C)$, which becomes a separable
Hilbert space and in addition a Banach algebra.  
In Benth, R\"udiger and S\"uss~\cite{BRS} and
Benth and Simonsen~\cite{BS} positive-definite Hilbert-Schmidt-valued Ornstein-Uhlenbeck processes have been defined and
studied in the context of stochastic volatility models in infinite dimensions. These volatility models become multilinear processes. 

Let $(E,\mathcal E)$ be a measurable space, and denote by ${\mathbb M}(E)$ the space of all finite signed measures. Equip
${\mathbb M}(E)$ with the total variation norm, $\|\nu\|_{\text{TV}}:=|\nu|(E)$ for $\nu\in {\mathbb M}(E)$. It is known that 
$({\mathbb M}(E),\|\cdot\|_{\text{TV}})$ is a 
Banach space. Define the convolution product of measures as
$$
\nu\ast\mu(A)=\int_{E\times E}\mathbf 1_A(x+y)\nu(dx)\mu(dy)=\int_E\nu(A-y)\mu(dy)
$$
for $\nu,\mu\in{\mathbb M}(E)$ and $A\in\mathcal E$. Since $\|\nu\ast\mu\|_{\text{TV}}\leq \|\nu\|_{\text{TV}}\|\mu\|_{\text{TV}}$, 
$({\mathbb M}(E),\|\cdot\|_{\text{TV}},\ast)$ is a 
Banach algebra which obviously is commutative. A polynomial $p_k\in\text{Pol}_k({\mathbb M}(E))$ will be of the form
$p_k(\mu)=\sum_{n=0}^k\nu_n\ast\mu^{\ast n}$ for $(\nu_n)_{n=0}^k\subset{\mathbb M}(E)$. These polynomials are built up from the monomials
$\mu^{\ast n}$. Cuchiero, Larsson and Svaluto-Ferro~\cite{CLSF} define polynomial processes on ${\mathbb M}(E)$ introducing {\it monomials} as follows:
Let $g:E^k\rightarrow\mathbb R$ be a continuous symmetric function. A monomial of degree $k\in\mathbb N$ is defined as
$$
{\mathbb M}(E)\ni\nu\mapsto\langle g,\nu^k\rangle:=\int_{E^k}g(x_1,\ldots,x_k)\nu(dx_1)\cdots\nu(dx_k)
$$ 
We notice that for any $A\in\mathcal E$,  we have that 
$$
\nu^{\ast k}(A)=\int_{E^k}\mathbb I_A(x_1+\cdots+x_k)\nu(dx_1)\cdots\nu(dx_k)=\langle\mathbb I_A(x_1+\cdots+x_k),\nu^k\rangle.
$$
Although the function $(x_1,\ldots,x_k)\mapsto \mathbb I_A(x_1+\cdots+x_k)$ is obviously not continuous, it is a 
bounded measurable symmetric
function which is linking our definition of polynomial processes on ${\mathbb M}(E)$ to the one of  Cuchiero, Larsson and Svaluto-Ferro~\cite{CLSF}.

\begin{appendix}
\section{Some auxiliary results on conditional expectation in Banach spaces}\label{sec:cond:exp}

Let $B$ be a Banach space over a field $\mathbb{F}$, which can be either $\mathbb{R}$ or $\mathbb{C}$. We denote the norm by $\|\cdot\|$ and by $\mathcal B(B)$ the Borel $\sigma$-algebra of $B$. 
Further, $L(B)$ denotes the space of bounded linear operators on $B$. Let $(\Omega,\mathcal F, \mathbb P)$ be a probability space equipped with a filtration $(\mathcal{F}_t)_{t\geq 0}$. Following the usual terminology (see e.g. 
Def.~1.4 in van Neerven \cite{vNeerven}), a $B$-valued random variable $X$ is a mapping from $\Omega$ into $B$ which is {\it strongly measurable}, that is, there exists a 
sequence of simple random variables $X_n:=\sum_{i=1}^n\mathbb I_{F_i} x_i$ where $F_i\in\mathcal F$, $x_i\in B, i=1,\ldots,n, n\in\mathbb N$ such that $X_n\rightarrow X$ in $B$ 
pointwise when $n\rightarrow\infty$. Here, we use the notation $\mathbb I_F:\Omega\rightarrow\{0,1\}$ as the indicator function on a set $F\in\mathcal F$. If $B$ is separable, then strong measurability is equivalent to measurability in the sense that $X^{-1}(A)\in\mathcal F$ for any $A\in\mathcal B(B)$. As a consequence of the approximation $X_n\rightarrow X$ a random variable $X$ takes values in the closed separable subspace $B_c:=\overline{ \text{{\tt span }} \cup_{n\in \mathbb{N}}\; \text{{\tt ran}} (X_n) }$, the closure of the subspace spanned by the ranges of the $X_n$.

In van Neerven~\cite{vNeerven}, a mapping $X:\Omega\rightarrow B$ is said to be {\it strongly $\mathbb P$-measurable} if there exists a 
sequence of
simple random variables $X_n$ such that $X_n\rightarrow X$ in $B$ $\mathbb P-a.s$ as $n\rightarrow\infty$. However, in view of Prop.~1.10
in van Neerven~\cite{vNeerven}, there exists a version $\widetilde{X}$ which is strongly measurable of any strongly $\mathbb P$-measurable random variable 
$X$, and vice versa. Thus, in our analysis we will always choose the strongly measurable version of a random variable, and therefore stick
to the notion of {\it strongly measurable} throughout. 

A random variable $X$ where $\mathbb E[\|X\|]<\infty$ is said to be Bochner integrable with respect to $\mathbb P$, and we 
define $\mathbb E[X]$ to be the Bochner integral 
(see e.g. Ch.~1\S 1 (J) of Dinculeanu~\cite{D})
$$
\mathbb E[X]:=\int_{\Omega}X\,d\mathbb P
$$ 
Moreover, $\mathbb E[X]\in B$ and $\|\mathbb E[X]\|\leq\mathbb E[\|X\|]$. Given a $\sigma$-algebra $\mathcal G\subset\mathcal F$ and
a Bochner integrable $B$-valued random variable, we define the conditional expectation $\mathbb E[X\,|\,\mathcal G]$ as the 
strongly $\mathcal G$-measurable random variable satisfying
\begin{equation}
\int_{G}\mathbb E[X\,|\,\mathcal G]\,d\mathbb P=\int_{G} X\,d\mathbb P
\end{equation} 
for all $G\in\mathcal G$ (see Def.~38 in Ch.~1\S 2 of Dinculeanu~\cite{D}). Thm.~50 and Prop.~37 in Ch.~1\S 2 of Dinculeanu~\cite{D} show 
that the conditional expectation exists and is unique $\mathbb P-a.s.$.

The next result shows that (conditional) expectation commutes with bounded linear operators:
\begin{lemma}\label{exchange:int:operator}
Suppose that $X$ is a $B$-valued random variable which is Bochner integrable and $\mathcal{L} \in L(B)$. Then
$\mathcal{L}X$ is a $B$-valued random variable which is Bochner integrable, $\E [ \mathcal{L} X] = \mathcal{L} \E [X] $ and 
$$
\E [ \mathcal{L} X\,|\,\mathcal{G} ] = \mathcal{L} \E [ X\,|\,\mathcal{G} ]
$$
for any $\mathcal{G} \subset \mathcal{F}$. 
\end{lemma}
\begin{proof}
 Since there exists a closed separable subspace $B_c\subseteq B$ with $\mathbb P[X\in B_c] = 1$, we know that also the range of $\mathcal{L}$ is separable. Now, the statement is given in Peszat and Zabczyk~\cite[Proposition 3.15(ii)]{peszat.zabczyk.07}.
\end{proof}

As the following Lemma shows, if $f:B\rightarrow B$ is continuous, then $f(X)$ is strongly measurable whenever $X$ is. This is
stronger than the first claim in the Lemma~\ref{exchange:int:operator} above. Indeed, the more general result holds: 
\begin{lemma}
\label{lemma:cont-preserve-measur}
If $f:B_1\rightarrow B_2$ is a continuous map between two Banach spaces $B_1$ and $B_2$, then $f(X)$ is a strongly measurable 
$B_2$-valued random variable for any strongly measurable $B_1$-valued random variable
$X$. 
\end{lemma}
\begin{proof}
For any sequence
of simple random variables $X_n$ converging to $X$, we have that $f(X_n)$ is converging to $f(X)$ by continuity. 
If $X_n=\sum_{i=1}^n\mathbb I_{F_i}x_i$
for disjoint sets $F_i\in\mathcal F$ (we can always do this by redefining the sum), we find $f(X_n)=\sum_{i=1}^n\mathbb I_{F_i}f(x_i)$, which shows that
$(f(X_n))_{n\in\mathbb N}$ is a sequence of simple random variables converging pointwise to $f(X)$. Hence, $f(X)$ is strongly measurable. 
\end{proof}
Choosing $B=B_1$ and $B_2=\mathbb R$, we find that $\|X\|$ is a measurable real-valued random variable when $X$ is a $B$-valued strongly 
measurable random variable. This is true since $x\mapsto\|x\|$ is a continuous map, as the triangle inequality shows that 
$\|x\|\leq \|x-y\|+\|y\|$ and $\|y\|\leq\|x-y\|+\|x\|$, and therefore $\vert\|x\|-\|y\|\vert\leq \|x-y\|$. 

Let us focus on the concept of independence for $B$-valued random variables. 
We recall that $X$ is independent of a $\sigma$-algebra 
$\mathcal G\subset\mathcal F$ if the two sets $X^{-1}(A)$ and $G\in\mathcal G$ are independent for all $A\in\mathcal B(B)$ and $G\in\mathcal G$. If $f:B\rightarrow B$ is a measurable map, it follows that $f(X)$ is independent of $\mathcal G$ whenever $X$ is independent of $\mathcal G$. This is so because for any $A\in\mathcal B(B)$, $(f(X))^{-1}(A)=X^{-1}(f^{-1}(A))$ and
$f^{-1}(A)\in\mathcal B(B)$ as $f$ is measurable. Moreover, if in addition $f$ is continuous, we see from 
Lemma~\ref{lemma:cont-preserve-measur} that $f(X)$ is a strongly measurable random variable being independent of $\mathcal G$. 
As a particular case, assume that $B$ is a Banach algebra, that is, $B$ is equipped with a multiplication operator 
$\cdot:B\times B\rightarrow B$ such that $(B,+,\cdot)$ is an associative $\mathbb F$-algebra and $\|x\cdot y\|\leq \|x\|\|y\|$ for
any $x,y\in B$. Consider the map
$f:B\ni x\mapsto x^2\in B$. By the norm property in a Banach algebra,
$$
\|x^2-y^2\|=\|x(x-y)+(x-y)y\|\leq(\|x\|+\|y\|)\|x-y\|,
$$ 
it follows that $f$ is continuous. Thus, we see that $X^2$ is independent of
$\mathcal G$ whenever $X$ is independent of $\mathcal G$. By induction, we have
that $X^k$ is independent of $\mathcal G$ whenever $X$ is independent of 
$\mathcal G$ for any $k\in\mathbb N$.   

We have the following result on conditional expectation of independent random
variables:
\begin{lemma} 
\label{lemma:cond-expect-independence}   
If $X$ is a $B$-valued Bochner-integrable random variable which is independent of 
the $\sigma$-algebra $\mathcal G\subset\mathcal F$, then
$$
\E[X\,|\,\mathcal G]=\E[X].
$$
\end{lemma}
\begin{proof}
Since $X$ is Bochner-integrable there is a separable closed subspace $B_c\subseteq B$ such that $\mathbb P[X\in B_c] = 1$. Consequently, we may assume that $B$ is separable. The statement is given in Peszat and Zabczyk~\cite[Proposition 3.15(v)]{peszat.zabczyk.07}.
\end{proof}

In our analysis, a "freezing property" of conditional expectation is important. To this end, we equip the product space
$B\times B$ of the Banach space $B$ with the max-norm, i.e.,
for any $x=(x_1,x_2)\in B\times B$, $\|x\|_2:=\max_{i=1,2}\|x_i\|$. Then, $B\times B$ is a Banach space again.
\begin{proposition}\label{feezing:lemma}
Let $\mathcal G\subset\mathcal F$ be a $\sigma$-algebra and $X$ a $B$-valued random variable independent of 
$\mathcal G$. Let further $Y$ be a $B$-valued random variable which is strongly $\mathcal G$-measurable and $f: B \times B \rightarrow B$ continuous. Assume 
$f(X,y)$ is Bochner integrable for every $y\in B$ and $y\mapsto\E[f(X,y)]$ is continuous, and moreover that there exist positive $\mathbb R$-valued random variables $Z,\widetilde{Z}$ with $\E[ Z  ]<\infty, \E[ \widetilde{Z}  ] <\infty $ and
\begin{align}
\| \E[f(X,y)]_{y=\widetilde{Y}} \| &\leq  Z \label{bound:1} \\
\| f(X,\widetilde{Y}) \| &\leq  \tilde{Z} \label{bound:2} ,
\end{align}
for every $B$-valued random variable $\widetilde{Y}$ such that $\| \widetilde{Y} \| \leq \| Y \|$ a.s., then
$$
\E[f(X,Y) \,|\,\mathcal G]=\E[f(X,y)]_{y=Y}.
$$
\end{proposition}
\begin{proof}
First, by Lemma~\ref{lemma:cont-preserve-measur}, we note that continuity of 
$f$ implies that $f(X,Y)$ and $f(X,y)$ are strongly measurable for any $y\in B$ (using $B_2=B$ and $B_1=B\times B$ or
$B_2=B$). Choosing $\widetilde{Y}=Y$ in \eqref{bound:2},
we find that $\E[\|f(X,Y)\|]\leq\E[\widetilde{Z}]<\infty$, and therefore $f(X,Y)$ is Bochner integrable. 

We need to show that $\E[ \mathbb I_{G} f(X,Y)  ] =\E[ \mathbb I_{G} \E [ f(X,y) ]_{y=Y} ]$ for any $G\in \mathcal G$.
 For this, let $G\in \mathcal G$ and recall that since $Y$ is strongly $\mathcal G$-measurable, there exists a sequence of 
$\mathcal G$-simple random variables $Y_n=\sum_{i=1}^ny_i\mathbb I_{A_i}$
with $y_i\in B$ and $A_i\in\mathcal G$ such that $Y_n\rightarrow Y$
pointwise. Moreover, by Thm.~6 in Ch.~1\S1 C of Dinculeanu~\cite{D} we can choose $Y_n$ such that  $\|Y_n\|\leq \|Y\|$. We also notice that the we can select the sets $A_1,\ldots,A_n$ to be disjoint, as we do. 
We see that
$$
f(X,Y_n)=\sum_{i=1}^n\mathbb I_{A_i}f(X,y_i)
$$
and by assumption \eqref{bound:2}, we calculate
\begin{align*}
\infty>\E[\|f(X,Y_n)\|]&=\E[\sum_{i=1}^n\mathbb I_{A_i}\|f(X,y_i)\|] \\
&=\sum_{i=1}^n\E[\mathbb I_{A_i}\|f(X,y_i)\|] \\
&=\sum_{i=1}^n
\mathbb P(A_i)\E[\|f(X,y_i)\|].
\end{align*}
In the last equality we used the fact that $\|f(X,y_i)\|$ is independent on $\mathcal G$, as $f$ and $\|\cdot\|$ are continuous functions and
$X$ is independent of $\mathcal G$ by assumption. In particular, this shows that $\E[\|f(X,y_i)\|]<\infty$, and hence $f(X,y_i)$ is
Bochner integrable. Therefore,
by Lemma~\ref{lemma:cond-expect-independence} it follows that 
\begin{equation}\label{y:fix:cond:exp}
\E[f(X,y_i)]=\E[f(X,y_i) \,|\,\mathcal G]
\end{equation} 
On the other hand,
$$
\E[f(X,y)]_{y=Y_n}=\sum_{i=1}^n\mathbb I_{A_i}\E[f(X,y_i)].
$$
Hence, $\E[f(X,y)]_{y=Y_n}$ is strongly $\mathcal G$-measurable and Bochner integrable since, from norm inequality of Bochner integrals and assumption \eqref{bound:2}
\begin{align*}
\E\left[\|\E[f(X,y)]_{y=Y_n}\|\right]&=\E\left[\sum_{i=1}^n\mathbb I_{A_i}\|\E[f(X,y_i)]\|\right] \\
&=\sum_{i=1}^n\mathbb P(A_i)
\|\E[f(X,y_i)]\| \\
&\leq\sum_{i=1}^n\mathbb P(A_i)\E[\|f(X,y_i)\|]\\
&=\E[\|f(X,Y_n)\|]<\infty.
\end{align*}
Thus, we calculate 
\begin{align*}
\mathbb E[\mathbb I_G \E [f(X,y) ]_{y=Y_n} ]&=\E \left[ \sum_{i=1}^n \mathbb I_G\mathbb I_{A_i}  \E [f(X,y_i) ]  \right] \\
&=\sum_{i=1}^n \E [ \mathbb I_{G \cap A_i} \E [ f(X,y_i) ]  ]   \\
&=\sum_{i=1}^n \E [ \mathbb I_{G \cap A_i}  f(X,y_i) ]   \\
&=\E [  \mathbb I_{G }  f(X,Y_n) ],
\end{align*} 
where the third equality uses (\ref{y:fix:cond:exp}) and the defining properties of conditional expectation. To see that in fact $\mathbb E[\mathbb I_G \E [f(X,y) ]_{y=Y} ]=\E [  \mathbb I_{G }  f(X,Y) ]$ we need to show that $\lim_{n\rightarrow \infty} \mathbb E[\mathbb I_G \E [f(X,y) ]_{y=Y_n} ]=\mathbb E[\mathbb I_G \E [f(X,y) ]_{y=Y} ]$ and $\lim_{n\rightarrow \infty} \E [  \mathbb I_{G }  f(X,Y_n) ] = \E [  \mathbb I_{G }  f(X,Y) ]$. For this note that since $f$ is continuous, $\mathbb I_Gf(X,Y_n)\rightarrow \mathbb I_Gf(X,Y)$ 
when $n\rightarrow\infty$. As $ \| Y_n \| \leq \|Y \|$, we have
$$
\|  \mathbb I_{G }  f(X,Y_n)   \| \leq \mathbb I_G \| f(X,Y_n)\| \leq  \tilde{Z}  
$$
and thus from dominated convergence it follows that
$$
\lim_{n\rightarrow\infty}\E[\mathbb I_{G }  f(X,Y_n)]=\E[\mathbb I_{G }  f(X,Y)] 
$$
For the other limit, we have from the continuity assumption on $y\mapsto\E[f(X,y)]$ that 
$\mathbb I_G\E[f(X,y)]_{y=Y_n}\rightarrow\mathbb I_G\E[f(X,y)]_{y=Y}$ pointwise in $B$. 
Furthermore, by assumption~\eqref{bound:1} 
$$
\| \mathbb I_G \E [f(X,y) ]_{y=Y_n} \| \leq \mathbb I_G \| \E [f(X,y) ]_{y=Y_n} \| \leq Z  
$$
Then, by dominated convergence, we obtain
$$
\lim_{n\rightarrow \infty} \mathbb E[\mathbb I_G \E [f(X,y) ]_{y=Y_n} ]=\mathbb E[\mathbb I_G \E [f(X,y) ]_{y=Y} ],
$$
and the proposition follows.
\end{proof}
Remark that condition \eqref{bound:1} is only used once, to obtain a uniform bound on $\| \E [f(X,y) ]_{y=Y_n} \|$. Further,
the continuity assumption on the function $y\mapsto\E[f(X,y)]$ is only used to have pointwise convergence. Both are used only in connection with concluding the final limit in the proof above. 


In the remainder of this appendix we will focus on the case when $B$ is a Banach algebra, and in particular
show the following fundamental property for conditional expectation:
Let $\mathcal G\subset\mathcal F$ be a $\sigma$-algebra and $Y$ a
$\mathcal{G}$-strongly measurable $B$-valued random variable, then
$$
\mathbb E[YX\,|\,\mathcal G]=Y\mathbb E[X\,|\,\mathcal G]\,, \mathbb E[XY\,|\,\mathcal G]=\mathbb E[X\,|\,\mathcal G] Y
$$
where $X$ is a $B$-valued random variable, with $X$, $YX$ and $XY$ such
that the conditional expectations are well-defined. 


First, we show a Lemma which will become convenient:
\begin{lemma}
\label{lem:bochner-factorization}
Let $(S,\Sigma,\mu)$ be a measure space and $B$ a Banach algebra. Suppose $F:(S,\Sigma,\mu)\rightarrow B$ is $\mu$-integrable
(that is, Bochner integrable with respect to $\mu$) and $g\in B$. Then
$gF$ and $Fg$ are $\mu$-integrable, and 
$$
\int_SgF(s)\,\mu(ds)=g\int_SF(s)\,\mu(ds)\,\qquad
\int_SF(s)g\,\mu(ds)=\int_SF(s)\,\mu(ds)g.
$$
\end{lemma}
\begin{proof}
  Define the continuous linear maps
  \begin{align*}
   \mathcal L_g&:B\rightarrow B, b\mapsto gb, \\
   \mathcal R_g&:B\rightarrow B, b\mapsto bg.
\end{align*}   
 Then we find
  $$ \int_S F(s)g \mu(ds) = \int_S \mathcal R_g F(s) \mu(ds) = \mathcal R_g\left(\int_S \mathcal F(s) \mu(ds)\right) = \int_S \mathcal F(s) \mu(ds)g. $$
  Similar with $\mathcal L_g$.
\end{proof}
The next Lemma shows that measurability is preserved under the product operation
in the Banach algebra:
\begin{lemma}
\label{lemma:product-measurability}
Let $\mathcal G\subset\mathcal F$ be a $\sigma$-algebra, and suppose that $Y$ and $Z$ are two strongly $\mathcal G$-measurable $B$-valued random 
variables and that $B$ is a Banach algebra. Then $YZ$ and $ZY$ are strongly $\mathcal G$-measurable $B$-valued random variables. 
\end{lemma}
\begin{proof}
The pair $(Y,Z)$ is strongly $B\times B$-measurable and the multiplication $\eta$ on $B$ is a continuous map from $B\times B$ to $B$. Thus, Lemma \ref{lemma:cont-preserve-measur} yields that $YZ = \eta(Y,Z)$ is strongly $\mathcal G$-measurable.
\end{proof}
We come to our final result of this appendix:
\begin{proposition}
\label{prop:fact-cond-expect}
Let $B$ be a Banach algebra and $\mathcal G\subset\mathcal F$ a $\sigma$-algebra. 
Assume that $X$ and $Y$ are two $B$-valued random variables where
$Y$ is strongly $\mathcal G$-measurable, $X$ is Bochner-integrable and $\|X\|\|Y\|$ is $\mathbb P$-integrable. Then
$$
\mathbb E[YX\,|\,\mathcal G]=Y\mathbb E[X\,|\,\mathcal G]\,,\qquad
\mathbb E[XY\,|\,\mathcal G]=\mathbb E[X\,|\,\mathcal G]Y.
$$
\end{proposition}
\begin{proof}
First, $\|XY\|\leq\|X\|\|Y\|$ and $\|YX\|\leq\|Y\|\|X\|$, so both $XY$ and $YX$ are Bochner integrable, and moreover, the conditional expectations of $XY$ and $YX$ with respect to $\mathcal G$ are well-defined. By assumption, the conditional expectation of $X$ with respect to $\mathcal G$ is also well-defined. 

By definition of the conditional expectation, $\mathbb E[X\,|\,\mathcal G]$
is strongly $\mathcal G$-measurable, and thus by Lemma~\ref{lemma:product-measurability}, 
$Y\mathbb E[X\,|\,\mathcal G]$ is $\mathcal{G}$-strongly measurable.
Let $Y_n=\sum_{i=1}^ny_i\mathbb I_{A_i}$
with $y_i\in B$ and $A_i\in\mathcal G$ for $i=1,\ldots,n$ be a sequence of $\mathcal G$-simple random variables 
such that $Y_n\rightarrow Y$
pointwise and by Thm.~6 in Ch.~1\S 1 C of Dinculeanu~\cite{D}, $\|Y_n\|\leq \|Y\|$. Let
$G\in\mathcal G$ be an arbitrary set. We find
\begin{align*}
\mathbb E[\mathbb I_G Y_nX]&=\sum_{i=1}^n\mathbb E[\mathbb I_G\mathbb I_{A_i}y_iX] \\
&=\sum_{i=1}^ny_i\mathbb E[\mathbb I_{G\cap A_i}X] \\
&=\sum_{i=1}^ny_i\mathbb E[\mathbb I_{G\cap A_i}\mathbb E[X\,|\,\mathcal G ]] \\
&=\sum_{i=1}^n\mathbb E[\mathbb I_G\mathbb I_{A_i}y_i\mathbb E[X\,|\,\mathcal G ]] \\
&=\mathbb E[\mathbb I_G Y_n\mathbb E[X\,|\,\mathcal G]]
\end{align*}  
In the second and fourth equalities we applied Lemma~\ref{lem:bochner-factorization}, and in
the third the definition of the conditional expectation. Now, 
$$
\|\mathbb I_G Y_nX\|\leq\mathbb I_G\|Y_n\|\|X\|
\leq\|Y\|\|X\|\in L^1(P)
$$
by assumption. Thus, as $Y_nX\rightarrow YX$, it follows by dominated convergence that  
$\mathbb E[\mathbb I_GY_nX]\rightarrow\mathbb E[\mathbb I_G YX]$. On the other hand,
$$
\|\mathbb I_G Y_n\mathbb E[X\,|\,\mathcal G]]\|\leq \mathbb I_G\|Y_n\|
\mathbb E[\|X\|\,|\,\mathcal G]\leq\|Y\|\mathbb E[\|X\|\,|\,\mathcal G]\,.
$$
$a.s.$, by Jensen's inequality (Property 44 in Ch.1\S 2 H of Dinculeanu~\cite{D}).
As $Y$ is strongly $\mathcal G$-measurable, it follows from Lemma~\ref{lemma:cont-preserve-measur} that 
$\|Y\|$ is a real-valued $\mathcal G$-measurable random variable. 
From the properties of conditional expectation for 
real-valued random variables 
$$
\mathbb E[\|Y\|\mathbb E[\|X\|\,|\,\mathcal G]]=\mathbb E[\mathbb E[\|Y\|\|X\|\,|\,\mathcal G]]=\mathbb E[\|Y\|\|X\|]<\infty
$$
by assumption. Hence, $\|Y\|\mathbb E[\|X\|\,|\,\mathcal G]$ is $P$-integrable, and since obviously it holds
pointwise that $\mathbb I_G Y_n\mathbb E[X\,|\,\mathcal G]\rightarrow\mathbb I_G Y\mathbb E[X\,|\,\mathcal G]$ we find by 
dominated convergence that
$\mathbb E[\mathbb I_G Y_n\mathbb E[X\,|\,\mathcal G]]\rightarrow\mathbb E
[\mathbb I_GY\mathbb E[X\,|\,\mathcal G]]$. We can therefore conclude
\begin{align*}
\mathbb E[\mathbb I_GY\mathbb E[X\,|\,\mathcal G]]=
\lim_{n\rightarrow\infty}\mathbb E[\mathbb I_GY_n\mathbb E[X\,|\,\mathcal G]] 
=\lim_{n\rightarrow\infty}\mathbb E[\mathbb I_GY_nX] 
=\mathbb E[\mathbb I_GYX].
\end{align*} 
Hence, the first result of the Proposition is proven. The second part follows in
the same manner. 
\end{proof}

\end{appendix}

\end{document}